\documentclass[10pt]{amsart}

\usepackage{amssymb}
\usepackage[leqno]{amsmath}
\usepackage{mathrsfs}
\usepackage{stmaryrd}
\usepackage{chemarrow}
\usepackage[norelsize]{algorithm2e}
\usepackage{enumerate}
\usepackage{graphicx}
\usepackage[pdftex,bookmarksnumbered,bookmarksopen,colorlinks,linkcolor=blue,anchorcolor=black,citecolor=blue,urlcolor=blue]{hyperref}
\usepackage[all]{xy}
\usepackage{tikz}
\usetikzlibrary{arrows}

\usepackage{multirow}
\usepackage[hyperpageref]{backref}
\usepackage{subfigure}
\usepackage{array}

  \newcounter{mnote}
  \let\oldmarginpar\marginpar
    \renewcommand\marginpar[1]{\-\oldmarginpar[\raggedleft\footnotesize #1]%
    {\raggedright\footnotesize #1}}

\newtheorem{theorem}{Theorem}[section]
\newtheorem{lemma}[theorem]{Lemma}
\newtheorem{corollary}[theorem]{Corollary}

\newtheorem{remark}[theorem]{Remark}

\newcommand{\curl}{\operatorname{curl}}
\renewcommand{\div}{\operatorname{div}}

\newcommand{\skw}{\operatorname{skw}}

\numberwithin{equation}{section}

\begin{document}
\title[Stabilization-Free VEM]{Virtual Element Methods Without Extrinsic Stabilization}
\author{Chunyu Chen}%
\address{Hunan Key Laboratory for Computation and Simulation in Science and Engineering; School of Mathematics and Computational Science, Xiangtan University, Xiangtan 411105, P.R.China }%
 \email{202131510114@smail.xtu.edu.cn}%
 \author{Xuehai Huang}%
 \address{School of Mathematics, Shanghai University of Finance and Economics, Shanghai 200433, China}%
 \email{huang.xuehai@sufe.edu.cn}%
\author{Huayi Wei}%
\address{Hunan Key Laboratory for Computation and Simulation in Science and Engineering; School of Mathematics and Computational Science, Xiangtan University, Xiangtan 411105, P.R.China }%
 \email{weihuayi@xtu.edu.cn}%

 \thanks{The second author is the corresponding author. The second author was supported by the National Natural Science Foundation of
China (NSFC) (Grant No. 12171300), and the Natural Science Foundation of Shanghai
(Grant No. 21ZR1480500).}
 \thanks{The first and third authors were supported by the National Natural
Science Foundation of China (NSFC) (Grant No. 12261131501, 11871413), and the construction of innovative provinces in Hunan Province (Grant No. 2021GK1010).}

\makeatletter
\@namedef{subjclassname@2020}{\textup{2020} Mathematics Subject Classification}
\makeatother
\subjclass[2020]{
65N12;   
65N22;   
65N30;   
}

\begin{abstract}
Virtual element methods (VEMs) without extrinsic stabilization in arbitrary degree of polynomial are developed for second order elliptic problems, including a nonconforming VEM and a conforming VEM in arbitrary dimension. The key is to construct local $H(\div)$-conforming macro finite element spaces such that the associated $L^2$ projection of the gradient of virtual element functions is computable, and the $L^2$ projector has a uniform lower bound on the gradient of virtual element function spaces in $L^2$ norm. Optimal error estimates are derived for these VEMs. Numerical experiments are provided to test the VEMs without extrinsic stabilization.
\end{abstract}

\keywords{Virtual element, stabilization, macro finite element, norm equivalence, error analysis}

\maketitle


\section{Introduction}
An additional stabilization term is usually required in the virtual element methods (VEMs) to ensure the coercivity of the discrete bilinear form~\cite{BeiraoBrezziCangianiManziniEtAl2013,BeiraoBrezziMariniRusso2014}.
The local stabilization term $S_K(\cdot, \cdot)$ has to satisfy 
$$
c_{*} |v|_{1,K}^2\leq S_K(v,v)\leq c^{*} |v|_{1,K}^2 
$$
for $v$ belongs to the non-polynomial subspace of the virtual element space, which influences the condition number of the stiffness matrix and brings in the pollution factor $\frac{\max\{1, c^*\}}{\min\{1, c_*\}}$ in the error estimates \cite{DassiMascotto2018,BeiraodaVeigaDassiRusso2017,Mascotto2018}.  
For the a posteriori error analysis on anisotropic polygonal meshes in \cite{AntoniettiBerroneBorioDAuriaEtAl2022}, the stabilization term dominates the error estimator, which makes the anisotropic a posteriori error estimator suboptimal. 
The stabilization term significantly affects the performance of the VEM for the Poisson eigenvalue problem \cite{BoffiGardiniGastaldi2020},
and improper choices of the stabilization term will produce useless results.
Special stabilization terms are designed for a nonlinear elasto-plastic deformation problem \cite{HudobivnikAldakheelWriggers2019} and an electromagnetic interface problem in three dimensions \cite{CaoChenGuo2023}, which are not easy to be extended to other problems. 
In short, the stabilization term has to be chosen carefully for different partial differential equations to make the VEM work well, which is arduous and will reduce its practicality.


A linear VEM without extrinsic stabilization, based on a higher order polynomial projection of the gradient of virtual element functions, is devised for the Poisson equation in two dimensions in \cite{BerroneBorioMarcon2021}, where the degree of polynomial used in the projection depends on the number of vertices of the polygon and generally the geometry of the polygon. 
Numerical examples in \cite{BerroneBorioMarcon2022} show that the VEM in \cite{BerroneBorioMarcon2021} outperforms the standard one in \cite{BeiraodaVeigaBrezziMariniRusso2016} for anisotropic elliptic problems on general convex polygonal meshes.
The idea in \cite{BerroneBorioMarcon2021} is difficult to be extended to construct VEMs without extrinsic stabilization in higher dimensions, and
the analysis is rather elaborate. 
This motivates us to construct VEMs without extrinsic stabilization in arbitrary dimension and arbitrary degree of polynomial in a unified way.


The key to construct VEMs without extrinsic stabilization is to find a finite-dimensional space $\mathbb{V}(K)$ for polytope $K$ and a projector $Q_K$ onto the space $\mathbb{V}(K)$ such that
\begin{enumerate}[(C1)]
\item It holds the norm equivalence 
\begin{equation}\label{intro:gradVknormequiv} 
\|Q_{K}\nabla v\|_{0,K}\eqsim \|\nabla v\|_{0,K} \quad \forall~v\in V_k(K)
\end{equation}
on shape function space $V_k(K)$ of virtual elements;
\item The projection $Q_{K}\nabla v$ is computable based on the degrees of freedom (DoFs) of virtual elements for $v\in V_k(K)$.
\end{enumerate}
The hidden constants in \eqref{intro:gradVknormequiv} are independent of the size of $K$, but depend on the degree of polynomials, and the chunkiness parameter and the geometric dimension of $K$; see Section~\ref{sec:meshcondition} for details.
We can choose $Q_{K}$ as the $L^2$-orthogonal projector with respect to the
inner product $(\cdot, \cdot)_K$. The norm equivalence
\eqref{intro:gradVknormequiv} implies that the space $\mathbb{V}(K)$ should be
sufficiently large compared with the virtual element space $V_k(K)$.  In
standard virtual element methods, $Q_{k-1}^{K}\nabla v$
\cite{BeiraodaVeigaBrezziMariniRusso2016} or $\nabla\Pi_k^{K}v$
\cite{BeiraoBrezziCangianiManziniEtAl2013,BeiraoBrezziMariniRusso2014,AhmadAlsaediBrezziMariniEtAl2013,AyusodeDiosLipnikovManzini2016}
are used, where $Q_{k-1}^{K}$ is the $L^2$-orthogonal projector onto the
$(k-1)$-th order polynomial space $\mathbb P_{k-1}(K; \mathbb{R}^d)$, and
$\Pi_k^{K}$ is the $H^1$ projection operator onto the $k$-th order polynomial
space $\mathbb P_{k}(K)$.  While only
\[
\|Q_{k-1}^{K}\nabla v\|_{0,K}\lesssim \|\nabla v\|_{0,K}, \quad \|\nabla\Pi_k^{K}v\|_{0,K}\lesssim \|\nabla v\|_{0,K}
\]
hold
rather than the norm equivalence \eqref{intro:gradVknormequiv}, then an additional stabilization term is usually required to ensure the coercivity of the discrete bilinear form.
To remove the additional stabilization term, based on a regular simplicial tessellation of polytope $K$, we employ $k$-th order or $(k-1)$-th order $H(\div)$-conforming macro finite elements as $\mathbb{V}(K)$ in this paper, and keep the virtual element space $V_k(K)$ as the usual ones.

We first construct $H(\div)$-conforming macro finite elements based on a
simplicial partition $\mathcal T_K$ of polytope $K$ in arbitrary dimension. The
shape function space $\mathbb{V}_{k}^{\rm div}(K)$ is a subspace of the $k$-th
order Brezzi-Douglas-Marini (BDM) element space on the simplicial partition
$\mathcal T_K$ for $k\geq1$ and the lowest order Raviart-Thomas (RT) element
space for $k=0$, with some constraints. To ensure the $L^2$ projection
$Q_{K,k}^{\div}\nabla v$ onto the space $\mathbb{V}_{k}^{\rm div}(K)$ is computable
for virtual element function $v\in V_k(K)$, we require that
$\div\boldsymbol{\phi}\in\mathbb P_{\max\{k-1,0\}}(K)$ and
$\boldsymbol{\phi}\cdot\boldsymbol{n}$ on each $(d-1)$-dimensional face of $K$
is a polynomial for $\boldsymbol{\phi}\in\mathbb{V}_{k}^{\rm div}(K)$.  Based on
these considerations and the direct decomposition of an $H(\div)$-conforming
macro finite element space related to $\mathbb{V}_{k}^{\rm div}(K)$, we propose
the unisolvent DoFs for the space $\mathbb{V}_{k}^{\rm div}(K)$, and establish the
$L^2$ norm equivalence.  By the way, we use the matrix-vector language to review
a conforming finite element for the differential $(d-2)$-form in
\cite{ArnoldFalkWinther2006,Arnold2018}.



By the aid of the projector $Q_{K,k}^{\div}$, we advance a nonconforming VEM and a conforming VEM without extrinsic stabilization for second order elliptic problems in arbitrary dimension.
Indeed, these VEMs can be equivalently recast as
primal mixed VEMs.
We prove the norm equivalence~\eqref{intro:gradVknormequiv} and the well-posedness of the VEMs without extrinsic stabilization, and derive the optimal error estimates. 

Numerical experiments are provided to test the convergence rate, the invertibility of the local stiffness matrices, the assembling time and the condition number of stiffness
matrix for the VEMs without extrinsic stabilization, which appear competitive with respect to other existing VEMs.

The idea on constructing VEMs without extrinsic stabilization in this paper is simple,
and can be extended to more VEMs and more partial differential equations.
Since there is no additional stabilization term, these VEMs will be preferred in the engineering community. More benefits of the VEMs without extrinsic stabilization will be the study of future works.
By the way, we refer to \cite{XuZhang2023} for a VEM without extrinsic stabilization on triangular meshes, \cite{CicuttinErnLemaire2019} for a hybrid high-order method, and 
\cite{YeZhang2020,AlTaweelWang2020,AlTaweelWang2020a,YeZhang2021,YeZhang2021a,AlTaweelWangYeZhang2021} for weak Galerkin finite element methods without extrinsic stabilization.

The rest of this paper is organized as follows. Notation and mesh conditions are presented in Section~\ref{sec:prelim}. 
In Section~\ref{sec:divmacrofem}, $H(\div)$-conforming macro finite elements in arbitrary dimension are constructed. A nonconforming VEM without extrinsic stabilization in arbitrary dimension is developed in Section \ref{sec:stabfreencfmvem}, and a conforming VEM without extrinsic stabilization in arbitrary dimension is devised in Section \ref{sec:stabfreecfmvem}.
Some numerical results are shown in Section \ref{sec:numericalexamps}.

\section{Preliminaries}\label{sec:prelim}

\subsection{Notation}
Let $\Omega\subset \mathbb{R}^d$ be a bounded
polytope. Given a bounded domain $K\subset\mathbb{R}^{d}$ and a
non-negative integer $m$, let $H^m(K)$ be the usual Sobolev space of functions
on $K$.
The corresponding norm and semi-norm are denoted respectively by
$\Vert\cdot\Vert_{m,K}$ and $|\cdot|_{m,K}$. By convention, let $L^2(K)=H^0(K)$. 
Let $(\cdot, \cdot)_K$ be the standard inner product on $L^2(K)$. If $K$ is $\Omega$, we abbreviate
$\Vert\cdot\Vert_{m,K}$, $|\cdot|_{m,K}$ and $(\cdot, \cdot)_K$ by $\Vert\cdot\Vert_{m}$, $|\cdot|_{m}$ and $(\cdot, \cdot)$,
respectively. Let $H_0^m(K)$ be the closure of $\mathcal C_{0}^{\infty}(K)$ with
respect to the norm $\Vert\cdot\Vert_{m,K}$, and $L_0^2(K)$ consist of all functions in $L^2(K)$ with zero mean value.
For integer $k\geq0$,
notation $\mathbb P_k(K)$ stands for the set of all
polynomials over $K$ with the total degree no more than $k$. Set $\mathbb P_{-1}(K)=\{0\}$.
For a Banach space $B(K)$, let $B(K; \mathbb{X}):=B(K)\otimes\mathbb{X}$ with $\mathbb{X}=\mathbb{R}^d$ and $\mathbb{K}$ being the set of antisymmetric matrices.
Denote by $Q_k^{K}$ the $L^2$-orthogonal projector onto $\mathbb P_k(K)$ or $\mathbb P_{k}(K; \mathbb{X})$.
Let $\skw\boldsymbol{\tau}:=(\boldsymbol{\tau}-\boldsymbol{\tau}^{\intercal})/2$ be the antisymmetric part of a tensor $\boldsymbol{\tau}$.
Denote by $\#S$ the number of elements in a finite set $S$.

Given a $d$-dimensional polytope $K$, let $\Delta_j(K)$ be the set of all $j$-dimensional faces of $K$ for $j=0,1,\ldots, d-1$. Set $\mathcal{F}(K):=\Delta_{d-1}(K)$ and $\mathcal{E}(K):=\Delta_{d-2}(K)$.
For $F\in\mathcal{F}(K)$, denote by $\boldsymbol{n}_{K,F}$ the unit outward normal vector to $\partial K$, which will be abbreviated as $\boldsymbol{n}_F$ or $\boldsymbol{n}$ if not causing any confusion.

Given a $d$-dimensional simplex $T$,
let $F_i\in\mathcal F(T)$ be the $(d-1)$-dimensional face opposite to vertex $\texttt{v}_i$, $\boldsymbol n_i$ be the unit outward normal to the face $F_i$, and $\lambda_i$ be the barycentric coordinate of the point $\boldsymbol x$ corresponding to the vertex $\texttt{v}_i$, for $i=0, 1, \cdots, d$.
Clearly $\{ \boldsymbol n_1, \boldsymbol n_2, \cdots, \boldsymbol n_d \}$ spans $\mathbb R^d$, and $\{\skw({\boldsymbol n_i\boldsymbol n_j^{\intercal}})\}_{1\leq i<j\leq d}$ spans the antisymmetric space $\mathbb K$. 
For $F\in\mathcal F(T)$, let $\mathcal{E}(F):=\{e\in\mathcal{E}(T): e\subset\partial F\}$.
For $e\in \mathcal E(F)$, denote by $\boldsymbol{n}_{F,e}$ be the unit vector outward normal to $\partial F$ but parallel to $F$.

Let $\{\mathcal {T}_h\}$ denote a family of partitions
of $\Omega$ into nonoverlapping simple polytopes with $h:=\max\limits_{K\in \mathcal {T}_h}h_K$
and $h_K:=\mbox{diam}(K)$.
Denote by $\mathcal F_h^r$ the set of all $(d-r)$-dimensional faces of the partition $\mathcal T_h$ for $r=1,\ldots,d$. Set $\mathcal F_h:=\mathcal F_h^1$ for simplicity. Let $\mathcal F_h^{\partial}$ be the subset of $\mathcal F_h$ including all $(d-1)$-dimensional faces on $\partial\Omega$. 
For any $F\in\mathcal{F}_h$,
let $h_F$ be its diameter and fix a unit normal vector $\boldsymbol{n}_F$.
For a piecewise smooth function $v$, define 
\[
\|v\|_{1,h}^2:=\sum_{K\in\mathcal T_h}\|v\|_{1,K}^2,\quad |v|_{1,h}^2:=\sum_{K\in\mathcal T_h}|v|_{1,K}^2.
\]

For domain $K$, we use $\boldsymbol{H}(\div, K)$ and $\boldsymbol{H}_0(\div, K)$ to denote the standard divergence vector spaces. 
For a smooth vector function $\boldsymbol{v}$, let $\nabla\boldsymbol{v}:=(\partial_iv_j)_{1\leq i,j\leq d}$.
On the face $F\in\mathcal F_h$, define the surface divergence
\[
\div_F\boldsymbol{v}=\div(\boldsymbol{v}-(\boldsymbol{v}\cdot\boldsymbol{n})\boldsymbol{n})=\div\boldsymbol{v}-\partial_n(\boldsymbol{v}\cdot\boldsymbol{n}).
\]
Define the surface gradient $\nabla_Fv:=\nabla v-(\partial_nv)\boldsymbol{n}$ for a smooth function $v$.

\subsection{Mesh conditions}\label{sec:meshcondition}
We impose the following conditions on the mesh $\mathcal T_h$ in this paper:
\begin{itemize}
 \item[(A1)] Each element $K\in \mathcal T_h$ and each face $F\in \mathcal F_h^r$ for $1\leq r\leq d-1$ is star-shaped with a uniformly bounded chunkiness parameter.

 \item[(A2)] There exists a shape-regular simplicial mesh $\mathcal T_h^*$ such that each $K\in \mathcal T_h$ is a union of some simplexes in $\mathcal T_h^*$. For polytope $K\in \mathcal T_h$, let $\mathcal T_K$ be the simplicial partition of $K$ induced from $\mathcal T_h^*$. Assume simplicial partition $\mathcal T_K$ is quasi-uniform.
\end{itemize}

For $K\in \mathcal T_h$, let $\boldsymbol{x}_K$ be the center of  the largest ball contained in $K$.
Throughout this paper, we use
``$\lesssim\cdots $" to mean that ``$\leq C\cdots$", where
$C$ is a generic positive constant independent of mesh size $h$, but may depend on the chunkiness parameter of the polytope, the degree of polynomials $k$, the dimension of space $d$, and the shape regularity and quasi-uniform constants of the virtual triangulation $\mathcal T^*_h$,
which may take different values at different appearances. Let $A\eqsim B$ mean $A\lesssim B$ and $B\lesssim A$.

Let $\mathcal{F}(\mathcal T_K)$ and $\mathcal{E}(\mathcal T_K)$ be the set of all $(d-1)$-dimensional faces and $(d-2)$-dimensional faces of the simplicial partition $\mathcal T_K$ respectively. Set 
\[
\mathcal{F}^{\partial}(\mathcal T_K):=\{F\in\mathcal{F}(\mathcal T_K): F\subset\partial K\},\quad \mathcal{E}^{\partial}(\mathcal T_K):=\{e\in\mathcal{E}(\mathcal T_K): e\subset\partial K\}.
\]

Hereafter we use $T$ to represent a simplex, and $K$ to denote a general polytope.

\section{$H(\div)$-Conforming Macro Finite Elements}\label{sec:divmacrofem}
In this section we will construct an $H(\div)$-conforming macro finite element space $\mathbb{V}_{k-1}^{\rm div}(K)$ and the corresponding degrees of freedoms in arbitrary dimension, and establish the $L^2$ norm equivalence for $\boldsymbol{\phi}\in\mathbb{V}_{k-1}^{\rm div}(K)$
\begin{equation*}
\|\boldsymbol{\phi}\|_{0,K}\eqsim h_K\|\div\boldsymbol{\phi}\|_{0,K} + \sup_{\boldsymbol{\psi}\in\div\mathring{\boldsymbol{V}}_{k}^{d-2}(K)}\frac{(\boldsymbol{\phi}, \boldsymbol{\psi})_K}{\|\boldsymbol{\psi}\|_{0,K}} +\sum_{F\in\mathcal F(K)}h_F^{1/2}\|\boldsymbol{\phi}\cdot\boldsymbol{n}\|_{0,F}.
\end{equation*}
The space $\mathbb{V}_{k-1}^{\rm div}(K)$ and its $L^2$ norm equivalence will be used to prove
the norm equivalence for the virtual element space
\begin{equation*}
\|Q_{K,k-1}^{\div}\nabla v\|_{0,K}\eqsim \|\nabla v\|_{0,K} \quad \forall~v\in V_k(K),
\end{equation*}
where $V_k(K)$ is the nonconforming virtual element space in Section \ref{sec:stabfreencfmvem}, and the conforming virtual element space in Section \ref{sec:stabfreecfmvem}. Here $Q_{K,k-1}^{\div}$ is the computable $L^2$ projector onto the space $\mathbb{V}_{k-1}^{\rm div}(K)$.

\subsection{$H(\div)$-conforming finite elements}
For a $d$-dimensional polytope $K\in \mathcal T_h$ and $k\geq2$, let 
\[
\boldsymbol{V}_{k-1}^{\mathrm{BDM}}(K):=\{\boldsymbol{\phi}\in\boldsymbol{H}(\div, K): \boldsymbol{\phi}|_{T}\in \mathbb P_{k-1}(T;\mathbb R^d) \textrm{ for each } T\in\mathcal T_K\}
\]
be the local Brezzi-Douglas-Marini (BDM) element space~\cite{BrezziDouglasMarini1986,BrezziDouglasDuranFortin1987,Nedelec:1986family}, whose degrees of freedom (DoFs) are given by \cite{ChenHuang2021divdiv}
\begin{align}
(\boldsymbol{v}\cdot\boldsymbol{n}, q)_F, & \quad q\in \mathbb P_{k-1}(F), F\in\mathcal F(T), \label{BDMdof1} \\
(\div\boldsymbol{v}, q)_T, & \quad q\in \mathbb P_{k-2}(T)/\mathbb R, \label{BDMdof2} \\
(\boldsymbol{v}, \boldsymbol{q})_T, & \quad \boldsymbol{q}\in \mathbb P_{k-3}(T;\mathbb K)\boldsymbol{x} \label{BDMdof3}
\end{align}
for $T\in\mathcal T_K$. Here $\mathbb P_{k-2}(T)/\mathbb R:=\mathbb P_{k-2}(T)\cap L_0^2(T)$, and $\mathbb P_{k-3}(T;\mathbb K)\boldsymbol{x}:=\{\boldsymbol{\tau}\boldsymbol{x}: \boldsymbol{\tau}\in \mathbb P_{k-3}(T;\mathbb K)\}$ with $\boldsymbol{x}\in T$ being the independent variable.
Define $\mathring{\boldsymbol{V}}_{k-1}^{\mathrm{BDM}}(K):=\boldsymbol{V}_{k-1}^{\mathrm{BDM}}(K)\cap \boldsymbol{H}_0(\div, K)$.

We also need the lowest order Raviart-Thomas (RT) element space~\cite{RaviartThomas1977,Nedelec1980}
\[
\boldsymbol{V}^{\mathrm{RT}}(K):=\{\boldsymbol{\phi}\in\boldsymbol{H}(\div, K): \boldsymbol{\phi}|_{T}\in \mathbb P_{0}(T;\mathbb R^d)+\boldsymbol{x}\mathbb P_{0}(T) \textrm{ for each } T\in\mathcal T_K\}.
\]
The DoFs are given by
\[
(\boldsymbol{v}\cdot\boldsymbol{n}, q)_F,  \quad q\in \mathbb P_{0}(F), F\in\mathcal F(T)
\]
for $T\in\mathcal T_K$.
Define $\mathring{\boldsymbol{V}}^{\mathrm{RT}}(K):=\boldsymbol{V}^{\mathrm{RT}}(K)\cap \boldsymbol{H}_0(\div, K)$.

\subsection{Finite element for differential $(d-2)$-form}
Now recall the finite element for differential $(d-2)$-form, i.e. $H\Lambda^{d-2}$-conforming finite element in \cite{ArnoldFalkWinther2006,Arnold2018}. We will present the finite element for differential $(d-2)$-form using the proxy of the differential form rather than the differential form itself as in \cite{ArnoldFalkWinther2006,Arnold2018}. 

By (3.5) in \cite{ChenHuang2021divdiv}, we have the direct decomposition
\begin{equation}\label{eq:Pkvectordecomp}  
\mathbb P_{k-1}(T;\mathbb R^{d})=\nabla\mathbb P_k(T)\oplus\mathbb P_{k-2}(T;\mathbb K)\boldsymbol{x}.
\end{equation}
Recall that \cite[(35)]{Chen;Huang:2020Finite}
\begin{equation}\label{eq:20230205}
\mathbb P_{k}(T)\cap\ker(I+\boldsymbol{x}\cdot\nabla)=\{0\}.
\end{equation}

\begin{lemma}\label{lem:skwgrad}
For $\boldsymbol{w}\in\mathbb P_{k-2}(T;\mathbb K)\boldsymbol{x}$ satisfying $(\skw\nabla\boldsymbol{w})\boldsymbol{x}=\boldsymbol{0}$, it holds $\boldsymbol{w}=\boldsymbol{0}$.
\end{lemma}
\begin{proof}
Since
\begin{equation*}
(\skw\nabla\boldsymbol{w})\boldsymbol{x}=\frac{1}{2}(\nabla\boldsymbol{w})\boldsymbol{x}-\frac{1}{2}(\nabla\boldsymbol{w})^{\intercal}\boldsymbol{x}=\frac{1}{2}\nabla(\boldsymbol{w}\cdot\boldsymbol{x})-\frac{1}{2}(I+\boldsymbol{x}\cdot\nabla)\boldsymbol{w},
\end{equation*}
we acquire from $\boldsymbol{w}\cdot\boldsymbol{x}=0$ that $(I+\boldsymbol{x}\cdot\nabla)\boldsymbol{w}=\boldsymbol{0}$, which together with \eqref{eq:20230205} implies $\boldsymbol{w}=\boldsymbol{0}$.
\end{proof}

\begin{lemma}
The polynomial complex
\begin{equation}\label{eq:polycomplex1}
\mathbb R\xrightarrow{}\mathbb P_k(T)\xrightarrow{\nabla}\mathbb P_{k-1}(T;\mathbb R^d)\xrightarrow{\skw\nabla}\mathbb P_{k-2}(T;\mathbb K)
\end{equation}
is exact.
\end{lemma}
\begin{proof}
Clearly \eqref{eq:polycomplex1} is a complex. It suffices to prove $\mathbb P_{k-1}(T;\mathbb R^d)\cap\ker(\skw\nabla)\subseteq\nabla\mathbb P_{k}(T)$.

For $\boldsymbol{v}\in\mathbb P_{k-1}(T;\mathbb R^d)\cap\ker(\skw\nabla)$, by decomposition \eqref{eq:Pkvectordecomp}, there exist $q\in \mathbb P_k(T)$ and $\boldsymbol{w}\in\mathbb P_{k-2}(T;\mathbb K)\boldsymbol{x}$ such that $\boldsymbol{v}=\nabla q+\boldsymbol{w}$. 
By $\skw\nabla\boldsymbol{v}=\boldsymbol{0}$, we get $\skw\nabla\boldsymbol{w}=\boldsymbol{0}$. 
Apply Lemma~\ref{lem:skwgrad} to derive $\boldsymbol{w}=\boldsymbol{0}$. Thus, $\boldsymbol{v}=\nabla q\in\nabla\mathbb P_k(T)$.
\end{proof}

\begin{lemma}
It holds the decomposition
\begin{equation}\label{eq:Pkskwtensordecomp}  
\mathbb P_{k-2}(T;\mathbb K)=\skw\nabla\mathbb P_{k-1}(T;\mathbb R^d)\oplus\big(\mathbb P_{k-2}(T;\mathbb K)\cap\ker(\boldsymbol{x})\big),
\end{equation}
where $\mathbb P_{k-2}(T;\mathbb K)\cap\ker(\boldsymbol{x}):=\{\boldsymbol{\tau}\in\mathbb P_{k-2}(T;\mathbb K): \boldsymbol{\tau}\boldsymbol{x}=\boldsymbol{0}\}$.
\end{lemma}
\begin{proof}
Thanks to decomposition \eqref{eq:Pkvectordecomp}, we have 
\[
\skw\nabla\mathbb P_{k-1}(T;\mathbb R^d)=\skw\nabla(\mathbb P_{k-2}(T;\mathbb K)\boldsymbol{x}).
\]
By Lemma~\ref{lem:skwgrad}, $\skw\nabla\mathbb P_{k-1}(T;\mathbb R^d)\cap\big(\mathbb P_{k-2}(T;\mathbb K)\cap\ker(\boldsymbol{x})\big)=\{\boldsymbol{0}\}$. Then we only need to check dimensions.
Due to complex \eqref{eq:polycomplex1},
\begin{equation}\label{eq:20220324-1}  
\dim\skw\nabla\mathbb P_{k-1}(T;\mathbb R^d)=\dim\mathbb P_{k-1}(T;\mathbb R^d)-\dim\nabla\mathbb P_k(T).
\end{equation}
On the other side, by space decomposition \eqref{eq:Pkvectordecomp},
\[
\dim\mathbb P_{k-2}(T;\mathbb K)\boldsymbol{x}=\dim\mathbb P_{k-1}(T;\mathbb R^{d})-\dim\nabla\mathbb P_k(T).
\]
Hence, $\dim\skw\nabla\mathbb P_{k-1}(T;\mathbb R^d)=\dim\mathbb P_{k-2}(T;\mathbb K)\boldsymbol{x}$, which yields \eqref{eq:Pkskwtensordecomp}.
\end{proof}


 By \eqref{eq:Pkskwtensordecomp} and \eqref{eq:20220324-1}, it follows
\begin{equation}\label{eq:20220324-2}
\dim\mathbb P_{k-2}(T;\mathbb K)\cap\ker(\boldsymbol{x})=\dim\mathbb P_{k-2}(T;\mathbb K)+\dim\nabla\mathbb P_k(T)- \dim\mathbb P_{k-1}(T;\mathbb R^d).
\end{equation}

With the decomposition \eqref{eq:Pkskwtensordecomp} and $\mathbb P_{k-1}(F;\mathbb R^{d-1})=\nabla_FP_k(F)\oplus\mathbb P_{k-2}(F;\mathbb K)\boldsymbol{x}$, we are ready to define the finite element for differential $(d-2)$-form. Take $\mathbb P_k(T;\mathbb K)$ as the space of shape functions. The degrees of freedom are given by
\begin{align}
((\boldsymbol{n}_1^e)^{\intercal}\boldsymbol{\tau}\boldsymbol{n}_2^e, q)_e, & \quad q\in \mathbb P_k(e), e\in\mathcal E(T), \label{eq:-2formfemdof1} \\
(\div_F(\boldsymbol{\tau}\boldsymbol{n}), q)_F, & \quad q\in \mathbb P_{k-1}(F)/\mathbb R, F\in\mathcal F(T), \label{eq:-2formfemdof21} \\
(\boldsymbol{\tau}\boldsymbol{n}, \boldsymbol{q})_F, & \quad \boldsymbol{q}\in \mathbb P_{k-2}(F;\mathbb K)\boldsymbol{x}, F\in\mathcal F(T), \label{eq:-2formfemdof22} \\
(\div\boldsymbol{\tau}, \boldsymbol{q})_T, & \quad \boldsymbol{q}\in\mathbb P_{k-3}(T;\mathbb K)\boldsymbol{x}, \label{eq:-2formfemdof31} \\
(\boldsymbol{\tau}, \boldsymbol{q})_T, & \quad \boldsymbol{q}\in \mathbb P_{k-2}(T;\mathbb K)\cap\ker(\boldsymbol{x}). \label{eq:-2formfemdof32}
\end{align}
In DoF \eqref{eq:-2formfemdof1}, $\boldsymbol{n}_1^e$ and  $\boldsymbol{n}_2^e$ are two unit normal vectors of $e$ satisfying $\boldsymbol{n}_1^e\cdot\boldsymbol{n}_2^e=0$. 
\begin{lemma}\label{lem:normalskwtensor}
For $e\in\mathcal E(T)$, let $\tilde{\boldsymbol{n}}_1$ and  $\tilde{\boldsymbol{n}}_2$ be another two unit normal vectors of $e$ satisfying $\tilde{\boldsymbol{n}}_1\cdot\tilde{\boldsymbol{n}}_2=0$. Then
\[
\skw(\tilde{\boldsymbol{n}}_1\tilde{\boldsymbol{n}}_2^{\intercal})=\pm\skw(\boldsymbol{n}_1^e(\boldsymbol{n}_2^e)^{\intercal}).
\]
\end{lemma}
\begin{proof}
Notice that there exists an orthonormal matrix $H\in\mathbb R^{2\times2}$ such that 
$(\tilde{\boldsymbol{n}}_1, \tilde{\boldsymbol{n}}_2)=(\boldsymbol{n}_1^e, \boldsymbol{n}_2^e)H$.
Then
\begin{align*}
2\skw(\tilde{\boldsymbol{n}}_1\tilde{\boldsymbol{n}}_2^{\intercal})&=\tilde{\boldsymbol{n}}_1\tilde{\boldsymbol{n}}_2^{\intercal}-\tilde{\boldsymbol{n}}_2\tilde{\boldsymbol{n}}_1^{\intercal}=(\tilde{\boldsymbol{n}}_1, \tilde{\boldsymbol{n}}_2)\begin{pmatrix}
0 & 1 \\
-1 & 0
\end{pmatrix}\begin{pmatrix}
\tilde{\boldsymbol{n}}_1^{\intercal} \\
\tilde{\boldsymbol{n}}_2^{\intercal}  
\end{pmatrix}
\\
&=(\boldsymbol{n}_1^e, \boldsymbol{n}_2^e)H\begin{pmatrix}
0 & 1 \\
-1 & 0
\end{pmatrix}H^{\intercal}(\boldsymbol{n}_1^e, \boldsymbol{n}_2^e)^{\intercal}.
\end{align*}
By a direct computation, $H\begin{pmatrix}
0 & 1 \\
-1 & 0
\end{pmatrix}H^{\intercal}={\rm det}(H)\begin{pmatrix}
0 & 1 \\
-1 & 0
\end{pmatrix}$. Hence
\[
2\skw(\tilde{\boldsymbol{n}}_1\tilde{\boldsymbol{n}}_2^{\intercal})=2{\rm det}(H)\skw(\boldsymbol{n}_1^e(\boldsymbol{n}_2^e)^{\intercal}),
\]
which ends the proof.
\end{proof}

\begin{lemma}\label{lem:-2formfemfaceunisolvence}
Let $\boldsymbol{\tau}\in\mathbb P_k(T;\mathbb K)$ and $F\in\mathcal F(T)$.
Assume the degrees of freedom \eqref{eq:-2formfemdof1}-\eqref{eq:-2formfemdof22} on $F$ vanish. 
Then $\boldsymbol{\tau}\boldsymbol{n}|_F=\boldsymbol{0}$.
\end{lemma}
\begin{proof}
Due to \eqref{eq:-2formfemdof1},  we get $(\boldsymbol{n}_1^e)^{\intercal}\boldsymbol{\tau}\boldsymbol{n}_2^e|_e=0$ on each $e\in\mathcal E(F)$, which together with Lemma~\ref{lem:normalskwtensor} indicates $\boldsymbol{n}_{F,e}^{\intercal}\boldsymbol{\tau}\boldsymbol{n}_F|_e=0$. By the unisolvence of BDM element on face $F$, cf. DoFs \eqref{BDMdof1}-\eqref{BDMdof3}, it follows from DoFs \eqref{eq:-2formfemdof21}-\eqref{eq:-2formfemdof22} that $\boldsymbol{\tau}\boldsymbol{n}|_F=\boldsymbol{0}$.
\end{proof}

\begin{lemma}\label{lem:PkKbubblecharac}
For $\boldsymbol{\tau}\in\mathbb P_k(T;\mathbb K)$, $\boldsymbol{\tau}\boldsymbol{n}|_{F_i}=\boldsymbol{0}$ for $i=1,\ldots, d$, if and only if 
\begin{equation}\label{eq:PkKbubblecharac}
\boldsymbol{\tau}=\sum_{1\leq i<j\leq d}\lambda_i\lambda_jq_{ij}\boldsymbol{N}_{ij}
\end{equation}
for some $q_{ij}\in\mathbb P_{k-2}(T)$. Here $\{\boldsymbol{N}_{ij}\}_{1\leq i<j\leq d}$ denotes the basis of $\mathbb K$ being dual to $\{\skw({\boldsymbol n_i\boldsymbol n_j^{\intercal}})\}_{1\leq i<j\leq d}$, i.e.,
\[
\boldsymbol{N}_{ij}:\skw({\boldsymbol n_l\boldsymbol n_m^{\intercal}})=\delta_{il}\delta_{jm},\quad 1\leq i<j\leq d,\; 1\leq l<m\leq d.
\]
\end{lemma}
\begin{proof}
For $1\leq l\leq d$ but $l\neq i, j$, by the definition of $\boldsymbol{N}_{ij}$, it holds $\boldsymbol{N}_{ij}\boldsymbol{n}_l=\boldsymbol{0}$. Hence, for $
\boldsymbol{\tau}=\sum\limits_{1\leq i<j\leq d}\lambda_i\lambda_jq_{ij}\boldsymbol{N}_{ij}
$, obviously we have $\boldsymbol{\tau}\boldsymbol{n}|_{F_i}=\boldsymbol{0}$ for $i=1,\ldots, d$.

On the other side, assume $\boldsymbol{\tau}\boldsymbol{n}|_{F_i}=\boldsymbol{0}$ for $i=1,\ldots, d$. Express $\boldsymbol{\tau}$ as 
\[
\boldsymbol{\tau}=\sum_{1\leq i<j\leq d}p_{ij}\boldsymbol{N}_{ij},
\]
where $p_{ij}=\boldsymbol{n}_i^{\intercal}\boldsymbol{\tau}\boldsymbol{n}_j\in\mathbb P_{k}(T)$.
Therefore, $p_{ij}|_{F_i}=p_{ij}|_{F_j}=0$, which ends the proof.
\end{proof}

\begin{lemma}
The degrees of freedom \eqref{eq:-2formfemdof1}-\eqref{eq:-2formfemdof32} are uni-solvent for $\mathbb P_k(T;\mathbb K)$.
\end{lemma}
\begin{proof}
By $\mathbb P_{k-1}(F;\mathbb R^{d-1})=\nabla_FP_k(F)\oplus\mathbb P_{k-2}(F;\mathbb K)\boldsymbol{x}$,
the number of degrees of freedom \eqref{eq:-2formfemdof21}-\eqref{eq:-2formfemdof22} is
$
(d^2+d){k+d-2\choose k-1} - (d+1){k+d-1\choose k}.
$
Using \eqref{eq:Pkvectordecomp} and \eqref{eq:20220324-2},
the number of degrees of freedom \eqref{eq:-2formfemdof1}-\eqref{eq:-2formfemdof32} is
\begin{align*}
&\frac{1}{2}(d^2+d){k+d-2\choose k} + (d^2+d){k+d-2\choose k-1} - (d+1){k+d-1\choose k} \\
&+ \frac{1}{2}(d^2+d){k+d-2\choose k-2}+{k+d\choose k}-(d+1){k+d-1\choose k-1} =\frac{1}{2}(d^2-d){k+d\choose k},
\end{align*}
which equals to $\dim\mathbb P_k(T;\mathbb K)$.

Assume $\boldsymbol{\tau}\in\mathbb P_k(T;\mathbb K)$ and all the degrees of freedom \eqref{eq:-2formfemdof1}-\eqref{eq:-2formfemdof32} vanish. 
It holds from Lemma~\ref{lem:-2formfemfaceunisolvence} that $\boldsymbol{\tau}\boldsymbol{n}|_{\partial T}=\boldsymbol{0}$. Noting that $\boldsymbol{\tau}$ is antisymmetric,  we also have $\boldsymbol{n}^{\intercal}\boldsymbol{\tau}|_{\partial T}=\boldsymbol{0}$. On each $F\in\mathcal F(T)$, it holds
\begin{equation}\label{eq:20220324-3}
\boldsymbol{n}^{\intercal}\div\boldsymbol{\tau}=\div(\boldsymbol{n}^{\intercal}\boldsymbol{\tau})= \div_F(\boldsymbol{n}^{\intercal}\boldsymbol{\tau})+\partial_n(\boldsymbol{n}^{\intercal}\boldsymbol{\tau}\boldsymbol{n})=\div_F(\boldsymbol{n}^{\intercal}\boldsymbol{\tau}).
\end{equation}
Hence $\boldsymbol{n}^{\intercal}\div\boldsymbol{\tau}|_{\partial T}=0$. 
Thanks to DoFs \eqref{BDMdof1}-\eqref{BDMdof3} for BDM element, we acquire from DoF \eqref{eq:-2formfemdof31} and $\div\div\boldsymbol{\tau}=0$ that $\div\boldsymbol{\tau}=\boldsymbol{0}$, which together with DoF \eqref{eq:-2formfemdof32} and decomposition \eqref{eq:Pkskwtensordecomp} gives
\[
(\boldsymbol{\tau}, \boldsymbol{q})_T=0  \quad \forall~\boldsymbol{q}\in \mathbb P_{k-2}(T;\mathbb K).
\]
Applying Lemma~\ref{lem:PkKbubblecharac}, $\boldsymbol{\tau}$ has the expression as in \eqref{eq:PkKbubblecharac}. Taking $\boldsymbol{q}=q_{ij}\skw({\boldsymbol n_i\boldsymbol n_j^{\intercal}})$ in the last equation for $1\leq i<j\leq d$, we get $q_{ij}=0$. Thus $\boldsymbol{\tau}=\boldsymbol{0}$.
\end{proof}

For polygon $K\in \mathcal T_h$, define the local finite element space for differential $(d-2)$-form  
\begin{align*}  
\boldsymbol{V}_{k}^{d-2}(K):=\{\boldsymbol{\tau}\in\boldsymbol{L}^2(K;\mathbb K)&: \boldsymbol{\tau}|_{T}\in \mathbb P_{k}(T;\mathbb K) \textrm{ for each } T\in\mathcal T_K, \\
&\;\;\textrm{ all the DoFs \eqref{eq:-2formfemdof1}-\eqref{eq:-2formfemdof22} are single-valued}\}.
\end{align*}
Thanks to Lemma~\ref{lem:-2formfemfaceunisolvence}, space $\boldsymbol{V}_{k}^{d-2}(K)$ is $H\Lambda^{d-2}$-conforming.
Define $\mathring{\boldsymbol{V}}_{k}^{d-2}(K):=\boldsymbol{V}_{k}^{d-2}(K)\cap \mathring{H}\Lambda^{d-2}(K)$, where $\mathring{H}\Lambda^{d-2}(K)$ is the subspace of $H\Lambda^{d-2}(K)$ with homogeneous boundary condition.
Notice that
 $\boldsymbol{V}_{k}^{d-2}(K)$ is the Lagrange element space for $d=2$,
and $\boldsymbol{V}_{k}^{d-2}(K)$ is the second kind N\'ed\'elec element space for $d=3$ \cite{Nedelec:1986family}. 

Recall the local finite element de Rham complexes in \cite{ArnoldFalkWinther2006,Arnold2018}.
For completeness, we will prove the exactness of these complexes.

\begin{lemma}
Let $k\geq2$. Finite element complexes
\begin{equation}\label{eq:localBDMfemderhamcomplex}
\boldsymbol{V}_{k}^{d-2}(K)\xrightarrow{\div\skw}\boldsymbol{V}_{k-1}^{\mathrm{BDM}}(K)\xrightarrow{\div} V_{k-2}^{L^2}(K)\to0,    
\end{equation}
\begin{equation}\label{eq:localRTfemderhamcomplex}
\boldsymbol{V}_{1}^{d-2}(K)\xrightarrow{\div\skw}\boldsymbol{V}^{\mathrm{RT}}(K)\xrightarrow{\div} V_{0}^{L^2}(K)\to0,    
\end{equation}
\begin{equation}\label{eq:localBDMfemderhamcomplex0}
\mathring{\boldsymbol{V}}_{k}^{d-2}(K)\xrightarrow{\div\skw}\mathring{\boldsymbol{V}}_{k-1}^{\mathrm{BDM}}(K)\xrightarrow{\div} \mathring{V}_{k-2}^{L^2}(K)\to0,    
\end{equation}
\begin{equation}\label{eq:localRTfemderhamcomplex0}
\mathring{\boldsymbol{V}}_{1}^{d-2}(K)\xrightarrow{\div\skw}\mathring{\boldsymbol{V}}^{\mathrm{RT}}(K)\xrightarrow{\div} \mathring{V}_{0}^{L^2}(K)\to0,    
\end{equation}
are exact, where $\mathring{V}_{k-2}^{L^2}(K):=V_{k-2}^{L^2}(K)/\mathbb R$, and
\[
V_{k-2}^{L^2}(K):=\{v\in L^2(K): v|_{T}\in \mathbb P_{k-2}(T) \textrm{ for each } T\in\mathcal T_K\}.
\]
\end{lemma}
\begin{proof}
We only prove complex \eqref{eq:localBDMfemderhamcomplex}, since the argument for the rest complexes is similar. Clearly \eqref{eq:localBDMfemderhamcomplex} is a complex. We refer to \cite[Section 4]{ChenHuang2021divX} for the proof of $\div\boldsymbol{V}_{k-1}^{\mathrm{BDM}}(K)=V_{k-2}^{L^2}(K)$.

Next prove $\boldsymbol{V}_{k-1}^{\mathrm{BDM}}(K)\cap\ker(\div)=\div\skw\boldsymbol{V}_{k}^{d-2}(K)$. For $\boldsymbol{v}\in\boldsymbol{V}_{k-1}^{\mathrm{BDM}}(K)\cap\ker(\div)$, by Theorem 1.1 in \cite{CostabelMcIntosh2010}, there exists $\boldsymbol{\tau}\in\boldsymbol{H}^1(K;\mathbb K)$ satisfying $\div\boldsymbol{\tau}=\div\skw\boldsymbol{\tau}=\boldsymbol{v}$. Let $\boldsymbol{\sigma}\in \boldsymbol{V}_{k}^{d-2}(K)$ be the nodal interpolation of $\boldsymbol{\tau}$ based on DoFs \eqref{eq:-2formfemdof1}-\eqref{eq:-2formfemdof32}. Thanks to DoF \eqref{eq:-2formfemdof1}, it follows from the integration by parts that
\[
(\div_F(\boldsymbol{\sigma}\boldsymbol{n}), 1)_F=(\boldsymbol{v}\cdot\boldsymbol{n}, 1)_F\quad\forall~F\in\mathcal F(\mathcal T_K),
\]
which together with \eqref{eq:20220324-3} and DoF \eqref{eq:-2formfemdof21} that
\[
(\boldsymbol{n}^{\intercal}\div\boldsymbol{\sigma}, q)_F = (\div_F(\boldsymbol{\sigma}\boldsymbol{n}), q)_F=(\boldsymbol{v}\cdot\boldsymbol{n}, q)_F\quad\forall~q\in \mathbb P_{k-1}(F),F\in\mathcal F(\mathcal T_K).
\]
Therefore, due to DoF \eqref{eq:-2formfemdof31} and the fact $\div\div\boldsymbol{\sigma}=\div\boldsymbol{v}=0$,
we acquire from the unisolvence of DoFs \eqref{BDMdof1}-\eqref{BDMdof3} for BDM element that $\boldsymbol{v}=\div\boldsymbol{\sigma} \in \div\skw\boldsymbol{V}_{k}^{d-2}(K)$.
\end{proof}

Note that $\div\skw=\curl$ for $d=2,3$.
For $k\geq1$, by finite element complexes \eqref{eq:localBDMfemderhamcomplex}-\eqref{eq:localRTfemderhamcomplex0}, we have
\begin{equation}\label{eq:20220324-4}
\dim\div\skw\boldsymbol{V}_{k}^{d-2}(K)-\dim\div\skw\mathring{\boldsymbol{V}}_{k}^{d-2}(K)={k+d-2\choose d-1}\#\mathcal F^{\partial}(\mathcal T_K)-1.
\end{equation}

\subsection{$H(\div)$-conforming macro finite element}
For each polygon $K\in \mathcal T_h$, 
define the shape function space
\[
\boldsymbol{V}_{k-1}^{\rm div}(K):=\{\boldsymbol{\phi}\in\boldsymbol{V}_{k-1}^{\mathrm{BDM}}(K): \div\boldsymbol{\phi}\in\mathbb P_{k-2}(K)\},
\]
for $k\geq 2$, and 
\[
\boldsymbol{V}_{0}^{\rm div}(K):=\{\boldsymbol{\phi}\in\boldsymbol{V}_{0}^{\mathrm{RT}}(K): \div\boldsymbol{\phi}\in\mathbb P_{0}(K)\}.
\]
Apparently $\mathbb P_{k-1}(K;\mathbb R^d)\subseteq\boldsymbol{V}_{k-1}^{\rm div}(K)$, $\boldsymbol{V}_{0}^{\rm div}(K)\cap\ker(\div)=\boldsymbol{V}_0^{\mathrm{RT}}(K)\cap\ker(\div)$, and $\boldsymbol{V}_{k-1}^{\rm div}(K)\cap\ker(\div)=\boldsymbol{V}_{k-1}^{\mathrm{BDM}}(K)\cap\ker(\div)$ for $k\geq2$.

In the following lemma we present a direct sum decomposition of space $\boldsymbol{V}_{k-1}^{\rm div}(K)$.

\begin{lemma}
For $k\geq1$,
it holds
\begin{equation}\label{eq:Vdivlocaldecomp}
\boldsymbol{V}_{k-1}^{\rm div}(K)=\div\skw \boldsymbol{V}_{k}^{d-2}(K) \oplus (\boldsymbol{x}-\boldsymbol{x}_K)\mathbb P_{\max\{k-2,0\}}(K).
\end{equation}
Then the complex
\begin{equation*}
\boldsymbol{V}_{k}^{d-2}(K)\xrightarrow{\div\skw}\boldsymbol{V}_{k-1}^{\rm div}(K)\xrightarrow{\div} \mathbb P_{\max\{k-2,0\}}(K)\to0    
\end{equation*}
is exact.
\end{lemma}
\begin{proof}
We only prove the case $k\geq2$, as the proof for case $k=1$ is similar.
Since $\div:(\boldsymbol{x}-\boldsymbol{x}_K)\mathbb P_{k-2}(K)\to\mathbb P_{k-2}(K)$ is bijective \cite[Lemma 3.1]{ChenHuang2021divdiv}, we have $\div\skw \boldsymbol{V}_{k}^{d-2}(K)\cap(\boldsymbol{x}-\boldsymbol{x}_K)\mathbb P_{k-2}(K)=\{\boldsymbol{0}\}$. Clearly $\div\skw \boldsymbol{V}_{k}^{d-2}(K) \oplus(\boldsymbol{x}-\boldsymbol{x}_K)\mathbb P_{k-2}(K)\subseteq \boldsymbol{V}_{k-1}^{\rm div}(K)$.

On the other side, for $\boldsymbol{\phi}\in \boldsymbol{V}_{k-1}^{\rm div}(K)$, by $\div\boldsymbol{\phi}\in \mathbb P_{k-2}(K)$, there exists a $q\in\mathbb P_{k-2}(K)$ such that $\div((\boldsymbol{x}-\boldsymbol{x}_K)q)=\div\boldsymbol{\phi}$, i.e. $\boldsymbol{\phi}-(\boldsymbol{x}-\boldsymbol{x}_K)q\in\boldsymbol{V}_{k-1}^{\rm div}(K)\cap\ker(\div)=\boldsymbol{V}_{k-1}^{\mathrm{BDM}}(K)\cap\ker(\div)$. Thanks to finite element complex \eqref{eq:localBDMfemderhamcomplex}, $\boldsymbol{\phi}-(\boldsymbol{x}-\boldsymbol{x}_K)q\in\div\skw \boldsymbol{V}_{k}^{d-2}(K)$. Thus \eqref{eq:Vdivlocaldecomp} follows.
\end{proof}

Based on the space decomposition \eqref{eq:Vdivlocaldecomp} and the degrees of freedom of BDM element, we propose the following DoFs for space $\boldsymbol{V}_{k-1}^{\rm div}(K)$
\begin{align}
    (\boldsymbol{\phi}\cdot\boldsymbol{n}, q)_F & \quad\forall~q\in\mathbb P_{k-1}(F) \textrm{ on each }  F\in\mathcal F^{\partial}(\mathcal T_K), \label{Vdivdof1}\\
    (\div\boldsymbol{\phi}, q)_K & \quad\forall~q\in\mathbb P_{\max\{k-2,0\}}(K)/\mathbb R, \label{Vdivdof2} \\
    (\boldsymbol{\phi}, \boldsymbol{q})_K & \quad\forall~\boldsymbol{q}\in \div\skw\mathring{\boldsymbol{V}}_{k}^{d-2}(K)=\div\mathring{\boldsymbol{V}}_{k}^{d-2}(K). \label{Vdivdof3}
\end{align}
\begin{lemma}
The set of DoFs \eqref{Vdivdof1}-\eqref{Vdivdof3} is uni-solvent for space $\boldsymbol{V}_{k-1}^{\rm div}(K)$.
\end{lemma}
\begin{proof}
By \eqref{eq:20220324-4} and \eqref{eq:Vdivlocaldecomp},
the number of DoFs \eqref{Vdivdof1}-\eqref{Vdivdof3} is
\begin{align*}
&\quad {k+d-2\choose d-1}\#\mathcal F^{\partial}(\mathcal T_K)+\dim\mathbb P_{\max\{k-2,0\}}(K)-1+\dim\div\skw\mathring{\boldsymbol{V}}_{k}^{d-2}(K) \\
&=\dim\div\skw\boldsymbol{V}_{k}^{d-2}(K)+\dim\mathbb P_{\max\{k-2,0\}}(K)=\dim\boldsymbol{V}_{k-1}^{\rm div}(K).
\end{align*}

Assume $\boldsymbol{\phi}\in\boldsymbol{V}_{k-1}^{\rm div}(K)$ and all the DoFs \eqref{Vdivdof1}-\eqref{Vdivdof3} vanish. By the vanishing DoF~\eqref{Vdivdof1}, $\boldsymbol{\phi}\in \boldsymbol{H}_0(\div, K)$ and $\div\boldsymbol{\phi}\in L_0^2(K)$. Then it follows from the vanishing DoF~\eqref{Vdivdof2} that $\div\boldsymbol{\phi}=0$. 
Thanks to the exactness of complexes \eqref{eq:localBDMfemderhamcomplex0}-\eqref{eq:localRTfemderhamcomplex0}, $\boldsymbol{\phi}\in\div\skw\mathring{\boldsymbol{V}}_{k}^{d-2}(K)$. Therefore $\boldsymbol{\phi}=\boldsymbol{0}$ holds from the vanishing DoF~\eqref{Vdivdof3}.
\end{proof}

\begin{remark}\rm
When $K$ is a simplex and $\mathcal T_K=\{K\}$, thanks to DoF \eqref{BDMdof3} for the BDM element, DoF \eqref{Vdivdof3} can be replaced by 
\[
(\boldsymbol{\phi}, \boldsymbol{q})_K \quad\forall~\boldsymbol{q}\in \mathbb P_{k-3}(K;\mathbb K)\boldsymbol{x}
\]
for $k\geq3$.
And DoF \eqref{Vdivdof3} disappears for $k=1$ and $k=2$.
\end{remark}

Next we consider the norm equivalence of space $\boldsymbol{V}_{k-1}^{\rm div}(K)$.

\begin{lemma}\label{lem:Vkm1divnormequivalence}
For $\boldsymbol{\phi}\in\boldsymbol{V}_{k-1}^{\rm div}(K)$, it holds the norm equivalence
\begin{equation}\label{eq:Vkm1divnormequivalence}
\|\boldsymbol{\phi}\|_{0,K}\eqsim h_K\|\div\boldsymbol{\phi}\|_{0,K} + \sup_{\boldsymbol{\psi}\in\div\mathring{\boldsymbol{V}}_{k}^{d-2}(K)}\frac{(\boldsymbol{\phi}, \boldsymbol{\psi})_K}{\|\boldsymbol{\psi}\|_{0,K}} +\sum_{F\in\mathcal F^{\partial}(\mathcal T_K)}h_F^{1/2}\|\boldsymbol{\phi}\cdot\boldsymbol{n}\|_{0,F}.
\end{equation}
\end{lemma}
\begin{proof}
By the inverse inequality \cite{Ciarlet1978,Verfuerth2013} (see also \cite[Lemma 10]{Huang2020}),
$$
h_K\|\div\boldsymbol{\phi}\|_{0,K}\lesssim  \|\div\boldsymbol{\phi}\|_{-1,K},
$$ 
where
$$
\|\div\boldsymbol{\phi}\|_{-1,K}=\sup_{v\in H_0^1(K)}\frac{(\div\boldsymbol{\phi}, v)_K}{|v|_{1,K}}=-\sup_{v\in H_0^1(K)}\frac{(\boldsymbol{\phi}, \nabla v)_K}{|v|_{1,K}} \leq \|\boldsymbol{\phi}\|_{0,K}.
$$
Then we have
\begin{equation}\label{eq:20230201}  
h_K\|\div\boldsymbol{\phi}\|_{0,K}\lesssim \|\boldsymbol{\phi}\|_{0,K}.
\end{equation}
For $F\in\mathcal F^{\partial}(\mathcal T_K)$, there exists a simplex $T\in\mathcal T_K$ satisfying $F\subset\partial T$, then apply the trace inequality \cite[Theorem 1.5.1.10]{Grisvard1985} (see also \cite[(2.18)]{BrennerSung2018}) and the inverse inequality to get
$$
h_F^{1/2}\|\boldsymbol{\phi}\cdot\boldsymbol{n}\|_{0,F} \lesssim \|\boldsymbol{\phi}\|_{0,T}+h_T|\boldsymbol{\phi}|_{1,T}\lesssim \|\boldsymbol{\phi}\|_{0,T}.
$$ 
This means
$$
\sum_{F\in\mathcal F^{\partial}(\mathcal T_K)}h_F^{1/2}\|\boldsymbol{\phi}\cdot\boldsymbol{n}\|_{0,F} \lesssim \sum_{F\in\mathcal F^{\partial}(\mathcal T_K)}\|\boldsymbol{\phi}\|_{0,T}\lesssim \|\boldsymbol{\phi}\|_{0,K}.
$$ 
Combining \eqref{eq:20230201}, the Cauchy-Schwarz inequality and the last inequality yields
$$
h_K\|\div\boldsymbol{\phi}\|_{0,K}+\sup_{\boldsymbol{\psi}\in\div\mathring{\boldsymbol{V}}_{k}^{d-2}(K)}\frac{(\boldsymbol{\phi}, \boldsymbol{\psi})_K}{\|\boldsymbol{\psi}\|_{0,K}} +\sum_{F\in\mathcal F^{\partial}(\mathcal T_K)}h_F^{1/2}\|\boldsymbol{\phi}\cdot\boldsymbol{n}\|_{0,F}\lesssim \|\boldsymbol{\phi}\|_{0,K}.
$$

Next we focus on the proof of the lower bound. Again we only prove the case $k\geq2$, whose argument can be applied to case $k=1$. Take $\boldsymbol{\phi}_1\in\boldsymbol{V}_{k-1}^{\mathrm{BDM}}(K)$ such that $(\boldsymbol{\phi}_1\cdot\boldsymbol{n})|_{\partial K}=(\boldsymbol{\phi}\cdot\boldsymbol{n})|_{\partial K}$, and all the DoFs \eqref{BDMdof1}-\eqref{BDMdof3} of $\boldsymbol{\phi}_1$ interior to $K$ equal to zero. 
By the norm equivalence on each simplex $T$ and the vanishing DoFs \eqref{BDMdof2}-\eqref{BDMdof3}, we get
\begin{align}\label{eq:20220324-5}
\|\boldsymbol{\phi}_1\|_{0,K}^2 = \sum_{T\in\mathcal T_K}\|\boldsymbol{\phi}_1\|_{0,T}^2 \eqsim \sum_{T\in\mathcal T_K}\sum_{F\in\mathcal F(T)}h_F\|\boldsymbol{\phi}_1\cdot\boldsymbol{n}\|_{0,F}^2 =\sum_{F\in\mathcal F^{\partial}(\mathcal T_K)}h_F\|\boldsymbol{\phi}\cdot\boldsymbol{n}\|_{0,F}^2.
\end{align}
Due to the vanishing DoF \eqref{BDMdof2}, it holds that $\div\boldsymbol{\phi}_1=Q_0^T(\div\boldsymbol{\phi}_1)$ for $T\in\mathcal T_K$.
Then apply the integration by parts and the Cauchy-Schwarz inequality to acquire
\begin{equation}\label{eq:202302011}
\|\div\boldsymbol{\phi}_1\|_{0,T}^2=\|Q_0^T(\div\boldsymbol{\phi}_1)\|_{0,T}^2
\leq\frac{1}{|T|}\sum_{F\in\mathcal F(T)\cap\mathcal F^{\partial}(\mathcal T_K)}|F|\|\boldsymbol{\phi}\cdot\boldsymbol{n}\|_{0,F}^2\;\;\forall~T\in\mathcal T_K.
\end{equation}
Now let $w\in H^1(K)\cap L_0^2(K)$ be the solution of 
\[
\begin{cases}
-\Delta w= \div(\boldsymbol{\phi}-\boldsymbol{\phi}_1)\quad\textrm{ in } K, \\
\;\;\partial_nw=0\qquad\qquad\quad\;\;\,\textrm{ on } \partial K.
\end{cases}
\]
The weak formulation is
\[
(\nabla w, \nabla v)_K=(\div(\boldsymbol{\phi}-\boldsymbol{\phi}_1), v)_K\quad\forall~v\in H^1(K)\cap L_0^2(K).
\]
Obviously we obtain from \eqref{eq:202302011} that
\begin{equation}\label{eq:202302031}  
\|\nabla w\|_{0,K}\lesssim h_K\|\div(\boldsymbol{\phi}-\boldsymbol{\phi}_1)\|_{0,K}\lesssim h_K\|\div\boldsymbol{\phi}\|_{0,K} +\sum_{F\in\mathcal F^{\partial}(\mathcal T_K)}h_F^{1/2}\|\boldsymbol{\phi}\cdot\boldsymbol{n}\|_{0,F}.
\end{equation}
Let $I_K^{\div}: \boldsymbol{H}_0(\div,K)\to\mathring{\boldsymbol{V}}_{k-1}^{\mathrm{BDM}}(K)$ be the local $L^2$-bounded commuting projection operator in \cite{ArnoldGuzman2021,FalkWinther2014}, then
\begin{equation}\label{eq:20230203}  
\|I_K^{\div}\boldsymbol{\psi}\|_{0,K} \lesssim \|\boldsymbol{\psi}\|_{0,K}\quad\forall~\boldsymbol{\psi}\in \boldsymbol{H}_0(\div,K),
\end{equation}
\[
\div(I_K^{\div}\boldsymbol{\psi})=\div\boldsymbol{\psi} \quad \textrm{ for } \boldsymbol{\psi}\in \boldsymbol{H}_0(\div,K) \textrm{ satisfying } \div\boldsymbol{\psi}\in V_{k-2}^{L^2}(K).
\]
Recall $
V_{k-2}^{L^2}(K)=\{v\in L^2(K): v|_{T}\in \mathbb P_{k-2}(T) \textrm{ for } T\in\mathcal T_K\}
$. Set $\boldsymbol{\phi}_2=-I_K^{\div}(\nabla w)\in\mathring{\boldsymbol{V}}_{k-1}^{\mathrm{BDM}}(K)$. We have
\begin{equation}\label{eq:20220324-6}
\div\boldsymbol{\phi}_2=-\div(I_K^{\div}(\nabla w))=-\Delta w=\div(\boldsymbol{\phi}-\boldsymbol{\phi}_1).
\end{equation}
It follows from \eqref{eq:20230203} and \eqref{eq:202302031} that
\begin{equation}\label{eq:20220324-7}
\|\boldsymbol{\phi}_2\|_{0,K}=\|I_K^{\div}(\nabla w)\|_{0,K}\lesssim \|\nabla w\|_{0,K}\lesssim h_K\|\div\boldsymbol{\phi}\|_{0,K} +\sum_{F\in\mathcal F^{\partial}(\mathcal T_K)}h_F^{1/2}\|\boldsymbol{\phi}\cdot\boldsymbol{n}\|_{0,F}.
\end{equation}
By \eqref{eq:20220324-6}, $\boldsymbol{\phi}-\boldsymbol{\phi}_1-\boldsymbol{\phi}_2\in\mathring{\boldsymbol{V}}_{k-1}^{\mathrm{BDM}}(K)\cap\ker(\div)$, which together the exactness of complex \eqref{eq:localBDMfemderhamcomplex0} indicates $\boldsymbol{\phi}-\boldsymbol{\phi}_1-\boldsymbol{\phi}_2\in\div\mathring{\boldsymbol{V}}_{k}^{d-2}(K)$. Hence
\begin{align*}  
\|\boldsymbol{\phi}\|_{0,K}&\lesssim \|\boldsymbol{\phi}_1\|_{0,K}+\|\boldsymbol{\phi}_2\|_{0,K}+\|\boldsymbol{\phi}-\boldsymbol{\phi}_1-\boldsymbol{\phi}_2\|_{0,K} \\
&\lesssim \|\boldsymbol{\phi}_1\|_{0,K}+\|\boldsymbol{\phi}_2\|_{0,K}+\sup_{\boldsymbol{\psi}\in\div\mathring{\boldsymbol{V}}_{k}^{d-2}(K)}\frac{(\boldsymbol{\phi}-\boldsymbol{\phi}_1-\boldsymbol{\phi}_2, \boldsymbol{\psi})_K}{\|\boldsymbol{\psi}\|_{0,K}} \\
&\lesssim \|\boldsymbol{\phi}_1\|_{0,K}+\|\boldsymbol{\phi}_2\|_{0,K}+\sup_{\boldsymbol{\psi}\in\div\mathring{\boldsymbol{V}}_{k}^{d-2}(K)}\frac{(\boldsymbol{\phi}, \boldsymbol{\psi})_K}{\|\boldsymbol{\psi}\|_{0,K}}.
\end{align*}
Finally, \eqref{eq:Vkm1divnormequivalence} holds from \eqref{eq:20220324-5} and \eqref{eq:20220324-7}.
\end{proof}

Let 
\[
\mathbb{V}_{k-1}^{\rm div}(K):=\{\boldsymbol{\phi}\in\boldsymbol{V}_{k-1}^{\rm div}(K): (\boldsymbol{\phi}\cdot\boldsymbol{n})|_F\in\mathbb P_{k-1}(F) \quad\forall~F\in\mathcal F(K)\}.
\]
On each face $F\in\mathcal F(K)$, $(\boldsymbol{\phi}\cdot\boldsymbol{n})|_F$ is a polynomial for $\boldsymbol{\phi}\in\mathbb{V}_{k-1}^{\rm div}(K)$ but $(\boldsymbol{\phi}\cdot\boldsymbol{n})|_F$ is a piecewise polynomial for $\boldsymbol{\phi}\in\boldsymbol{V}_{k-1}^{\rm div}(K)$.
Due to DoFs \eqref{Vdivdof1}-\eqref{Vdivdof3} for $\boldsymbol{V}_{k-1}^{\rm div}(K)$, a set of unisolvent DoFs for $\mathbb{V}_{k-1}^{\rm div}(K)$ is 
\begin{align}
    (\boldsymbol{\phi}\cdot\boldsymbol{n}, q)_F & \quad\forall~q\in\mathbb P_{k-1}(F) \textrm{ on each }  F\in\mathcal F(K), \label{Vpdivdof1}\\
    (\div\boldsymbol{\phi}, q)_K & \quad\forall~q\in\mathbb P_{\max\{k-2,0\}}(K)/\mathbb R, \label{Vpdivdof2} \\
    (\boldsymbol{\phi}, \boldsymbol{q})_K & \quad\forall~\boldsymbol{q}\in \div\mathring{\boldsymbol{V}}_{k}^{d-2}(K). \label{Vpdivdof3}
\end{align}

As an immediate result of Lemma~\ref{lem:Vkm1divnormequivalence}, we get the following norm equivalence of space $\mathbb{V}_{k-1}^{\rm div}(K)$.
\begin{corollary}
For $\boldsymbol{\phi}\in\mathbb{V}_{k-1}^{\rm div}(K)$, it holds the norm equivalence
\begin{equation}\label{eq:Vkm1divnormequiv}
\|\boldsymbol{\phi}\|_{0,K}\eqsim h_K\|\div\boldsymbol{\phi}\|_{0,K} + \sup_{\boldsymbol{\psi}\in\div\mathring{\boldsymbol{V}}_{k}^{d-2}(K)}\frac{(\boldsymbol{\phi}, \boldsymbol{\psi})_K}{\|\boldsymbol{\psi}\|_{0,K}} +\sum_{F\in\mathcal F(K)}h_F^{1/2}\|\boldsymbol{\phi}\cdot\boldsymbol{n}\|_{0,F}.
\end{equation}
\end{corollary}

For later use, let $Q_{K,k-1}^{\div}$ be the $L^2$-orthogonal projection operator onto $\mathbb{V}_{k-1}^{\rm div}(K)$ with respect to the inner product $(\cdot, \cdot)_K$.
Introduce the discrete spaces
\begin{align*}
\mathbb{V}_{h,k-1}^{\rm div}&:=\{\boldsymbol{\phi}_h\in \boldsymbol{L}^2(\Omega;\mathbb R^d): \boldsymbol{\phi}_h|_K\in \mathbb{V}_{k-1}^{\rm div}(K) \textrm{ for each } K\in\mathcal T_h\}, \\
\mathbb P_l(\mathcal T_h)&:=\{q_h\in L^2(\Omega): q_h|_K\in \mathbb P_l(K) \textrm{ for each } K\in\mathcal T_h\}
\end{align*}
with non-negative integer $l$.
For $\boldsymbol{\phi}\in\boldsymbol{L}^2(\Omega;\mathbb R^d)$, let $Q_{h,k-1}^{\div}\boldsymbol{\phi}\in\mathbb{V}_{h,k-1}^{\rm div}$ be determined by $(Q_{h,k-1}^{\div}\boldsymbol{\phi})|_K=Q_{K,k-1}^{\div}(\boldsymbol{\phi}|_K)$ for each $K\in\mathcal T_h$. For $v\in L^2(\Omega)$, let $Q_h^lv\in \mathbb P_l(\mathcal T_h)$ be determined by $(Q_h^{l}v)|_K=Q_{l}^K(v|_K)$ for each $K\in\mathcal T_h$. For simplicity, the vector version of $Q_h^l$ is still denoted by $Q_h^l$. And we abbreviate $Q_h^k$ as $Q_h$ if $l=k$.

\section{Nonconforming virtual element method without extrinsic stabilization}\label{sec:stabfreencfmvem}


In this section we will develop a nonconforming VEM without extrinsic stabilization for the second order elliptic problem in arbitrary dimension
\begin{equation}\label{eq:ellipitc2ndproblem}
\begin{cases}
-\Delta u + \alpha u=f & \textrm{ in } \Omega,\\
\qquad\qquad u=0&\textrm{ on } \partial\Omega,
\end{cases}
\end{equation}
where $\Omega\subseteq\mathbb R^d$ is a bounded polygon, $f\in L^2(\Omega)$ and $\alpha$ is a nonnegative constant. The weak formulation of problem \eqref{eq:ellipitc2ndproblem} is to find $u\in H_0^1(\Omega)$ such that
\begin{equation}\label{eq:ellipitc2ndproblemweakform}
a(u,v)=(f,v)\quad\forall~v\in H_0^1(\Omega),
\end{equation}
where the bilinear form $a(u, v):=(\nabla_h u, \nabla_h v)+\alpha(u,v)$ with $\nabla_h$ being the piecewise counterpart of $\nabla$ with respect to $\mathcal T_h$.

\subsection{$H^1$-nonconforming virtual element}
Several $H^1$-nonconforming virtual elements have been developed in \cite{AyusodeDiosLipnikovManzini2016,CangianiManziniSutton2017,ChenHuang2020ncvem,Huang2020}.
In this paper we adopt those in \cite{CangianiManziniSutton2017,ChenHuang2020ncvem}.
The degrees of freedom are given by
\begin{align}
\frac{1}{|F|}(v, \phi_i^F)_F, & \quad i=1,\ldots, \dim\mathbb P_{k-1}(F), F\in\mathcal F(K), \label{eq:vemdof1}\\
\frac{1}{|K|}(v, \phi_i^K)_K, & \quad i=1,\ldots, \dim\mathbb P_{k-2}(K), \label{eq:vemdof2}
\end{align}
where $\{\phi_i^F\}_{i=1}^{\dim\mathbb P_{k-1}(F)}$ is a basis of $\mathbb P_{k-1}(F)$, and $\{\phi_i^K\}_{i=1}^{\dim\mathbb P_{k-2}(K)}$ a basis of $\mathbb P_{k-2}(K)$.

To define the space of shape functions, we need a local $H^1$ projection operator $\Pi_k^K: H^1(K)\to\mathbb P_k(K)$: given $v\in H^1(K)$, let $\Pi_k^Kv\in\mathbb P_k(K)$ be the solution of the problem 
\begin{align}
(\nabla\Pi_k^Kv, \nabla q)_K&=(\nabla v, \nabla q)_K\quad  \forall~q\in \mathbb P_k(K), \label{eq:projlocal2d1}\\
\int_{\partial K}\Pi_k^Kv\,{\rm d}s&=\int_{\partial K}v\,{\rm d}s. \label{eq:projlocal2d2}
\end{align}
It holds
\begin{equation}\label{eq:localproj2dprop1}
\Pi_k^Kq=q \quad\forall~q\in\mathbb P_k(K).    
\end{equation}

With the help of operator $\Pi_k^K$, the space of shape functions is defined as
\begin{align*}
V_k(K):=\{v\in H^1(K) &: \Delta v\in\mathbb P_{k}(K),\, \partial_nv|_F\in\mathbb P_{k-1}(F)\textrm{ for each face }F\in\mathcal F(K), \\
&\qquad\qquad\qquad\quad\;\, \textrm{and } (v-\Pi_k^Kv, q)_K=0\quad\forall~q\in \mathbb P_{k-2}^{\perp}(K)\},
\end{align*}
where 
$\mathbb P_{k-2}^{\perp}(K)$
means the orthogonal complement space of $\mathbb P_{k-2}(K)$ in $\mathbb P_{k}(K)$ with respect to the inner product $(\cdot, \cdot)_K$.
Due to \eqref{eq:localproj2dprop1}, it holds $\mathbb P_k(K)\subseteq V_k(K)$.
DoFs~\eqref{eq:vemdof1}-\eqref{eq:vemdof2} are uni-solvent for the shape function space $V_k(K)$.

For $v\in V_k(K)$, the $H^1$ projection $\Pi_k^Kv$ is computable using DoFs \eqref{eq:vemdof1}-\eqref{eq:vemdof2}, and the $L^2$ projection
\begin{equation}\label{eq:QKPik}  
Q_k^Kv= \Pi_k^Kv + Q_{k-2}^Kv-Q_{k-2}^K\Pi_k^Kv
\end{equation}
is also computable using DoFs \eqref{eq:vemdof1}-\eqref{eq:vemdof2}. 

We will prove the inverse inequality and the norm equivalence for the virtual element space $V_k(K)$.
\begin{lemma}\label{lem:veminverse}
It holds the inverse inequality
\begin{equation}\label{eq:veminverse}
|v|_{1,K}\lesssim h_K^{-1}\|v\|_{0,K}\quad\forall~v\in V_k(K).  
\end{equation}
\end{lemma}
\begin{proof}
By (A.4) with $m=j=1$ in \cite{ChenHuang2020ncvem}, it follows that
\[
h_K^{1/2}\|\partial_nv\|_{0,\partial K}\lesssim |v|_{1,K}+h_K\|\Delta v\|_{0,K}.
\]
Then apply (A.3) in \cite{ChenHuang2020ncvem} to get
\[
h_K\|\Delta v\|_{0,K}+h_K^{1/2}\|\partial_nv\|_{0,\partial K}\lesssim |v|_{1,K}+h_K\|\Delta v\|_{0,K}\lesssim |v|_{1,K}.
\]
Employing the integration by parts and the Cauchy-Schwarz inequality, we have
\[
|v|_{1,K}^2=-(\Delta v,v)_K+(\partial_nv,v)_{\partial K}\leq\|\Delta v\|_{0,K}\|v\|_{0,K}+\|\partial_nv\|_{0,\partial K}\|v\|_{0,\partial K}.
\]
Combining the last two inequalities gives
\[
|v|_{1,K}\lesssim h_K^{-1}\|v\|_{0,K}+h_K^{-1/2}\|v\|_{0,\partial K},
\]
which together with the multiplicative trace inequality and the Young's inequality yields \eqref{eq:veminverse}.
\end{proof}

\begin{lemma}
For $v\in V_k(K)$, we have
\begin{align}
\|\Pi_k^Kv\|_{0,K}^2+h_K^2|\Pi_k^Kv|_{1,K}^2&\lesssim \|Q_{k-2}^Kv\|_{0,K}^2+\sum_{F\in\mathcal F(K)}h_F\|Q_{k-1}^Fv\|_{0,F}^2, \label{eq:VEMnormequivalencePiK}\\
\|Q_k^Kv\|_{0,K}^2&\lesssim \|Q_{k-2}^Kv\|_{0,K}^2+\sum_{F\in\mathcal F(K)}h_F\|Q_{k-1}^Fv\|_{0,F}^2. \label{eq:VEMnormequivalenceQK}
\end{align}
\end{lemma}
\begin{proof}
We get from \eqref{eq:projlocal2d1} and the integration by parts that
\begin{align*}
|\Pi_k^Kv|_{1,K}^2&=(\nabla v, \nabla\Pi_k^Kv)_K=-(v, \Delta\Pi_k^Kv)_K+(v, \partial_n(\Pi_k^Kv))_{\partial K} \\
&=-(Q_{k-2}^Kv, \Delta\Pi_k^Kv)_K+\sum_{F\in\mathcal F(K)}(Q_{k-1}^Fv, \partial_n(\Pi_k^Kv))_{F} \\
&\leq\|Q_{k-2}^Kv\|_{0,K}\|\Delta\Pi_k^Kv\|_{0,K}+\sum_{F\in\mathcal F(K)}\|Q_{k-1}^Fv\|_{0,F}\|\partial_n(\Pi_k^Kv)\|_{0,F},
\end{align*}
which combined with both $H^1$-$L^2$ and $L^2$ boundary-$L^2$ bulk inverse inequalities for polynomials implies
\[
h_K^2|\Pi_k^Kv|_{1,K}^2\lesssim \|Q_{k-2}^Kv\|_{0,K}^2+\sum_{F\in\mathcal F(K)}h_F\|Q_{k-1}^Fv\|_{0,F}^2. 
\]
Thanks to the Poincar\'e-Friedrichs inequality \cite{Necas1967} 
and \eqref{eq:projlocal2d2},
\begin{align*}  
\|\Pi_k^Kv\|_{0,K}^2&\lesssim h_K|\Pi_k^Kv|_{1,K}^2+h_K^{2-d}\left|\int_{\partial K}v\,{\rm d}s\right|^2 \\
&=h_K^2|\Pi_k^Kv|_{1,K}^2+h_K^{2-d}\left|\sum_{F\in\mathcal F(K)}\int_{F}Q_{0}^Fv\,{\rm d}s\right|^2 \\
&\lesssim h_K^2|\Pi_k^Kv|_{1,K}^2+\sum_{F\in\mathcal F(K)}h_F\|Q_{0}^Fv\|_{0,F}^2.
\end{align*}
Hence \eqref{eq:VEMnormequivalencePiK} follows from the last two inequalities.

Finally, \eqref{eq:VEMnormequivalenceQK} holds from \eqref{eq:QKPik} and \eqref{eq:VEMnormequivalencePiK}.
\end{proof}

\begin{lemma}
It holds the norm equivalence
\begin{equation}\label{eq:Vknormequivalence}
h_K^2|v|_{1,K}^2\lesssim\|v\|_{0,K}^2\eqsim \|Q_{k-2}^Kv\|_{0,K}^2+\sum_{F\in\mathcal F(K)}h_F\|Q_{k-1}^Fv\|_{0,F}^2 \quad\forall~v\in V_k(K).
\end{equation}  
\end{lemma}
\begin{proof}
Since $\Delta v\in\mathbb P_{k}(K)$ and $\partial_nv|_F\in\mathbb P_{k-1}(F)$, we get from the integration by parts that
\begin{align*}
|v|_{1,K}^2&=-(\Delta v,Q_k^Kv)_K+\sum_{F\in\mathcal F(K)}(\partial_nv,Q_{k-1}^Fv)_{F} \\
&\leq\|\Delta v\|_{0,K}\|Q_k^Kv\|_{0,K}+\sum_{F\in\mathcal F(K)}\|\partial_nv\|_{0,F}\|Q_{k-1}^Fv\|_{0,F}.
\end{align*}
Applying the similar argument as in Lemma \ref{lem:veminverse}, we obtain
\[
h_K^2|v|_{1,K}^2\lesssim \|Q_k^Kv\|_{0,K}^2+\sum_{F\in\mathcal F(K)}h_F\|Q_{k-1}^Fv\|_{0,F}^2.
\]
Then it follows from \eqref{eq:VEMnormequivalenceQK} that
\[
\|v\|_{0,K}^2\lesssim \|Q_{k-2}^Kv\|_{0,K}^2+\sum_{F\in\mathcal F(K)}h_F\|Q_{k-1}^Fv\|_{0,F}^2.
\]
The other side $\|Q_{k-2}^Kv\|_{0,K}^2+\sum\limits_{F\in\mathcal F(K)}h_F\|Q_{k-1}^Fv\|_{0,F}^2\lesssim \|v\|_{0,K}^2$ holds from the trace inequality and the inverse inequality \eqref{eq:veminverse}.
\end{proof}

\subsection{Local inf-sup condition and norm equivalence}
With the help of the macro element space $\mathbb{V}_{k-1}^{\rm div}(K)$, we will present a norm equivalence for space $\nabla V_k(K)$, which is vitally important to design virtual element methods without extrinsic stabilization. 
\begin{lemma}\label{lem:gradVknormequivalence}
It holds the inf-sup condition
\begin{equation}\label{eq:localdiscreteinfsup}    
\|\nabla v\|_{0,K}\leq C_e\sup_{\boldsymbol{\phi}\in\mathbb{V}_{k-1}^{\rm div}(K)}\frac{(\boldsymbol{\phi}, \nabla v)_K}{\|\boldsymbol{\phi}\|_{0,K}} \quad \forall~v\in V_k(K),
\end{equation}
where the constant $C_e\geq1$ is independent of the mesh size $h_K$, but depends on the chunkiness parameter of the polytope, the degree of polynomials $k$, the dimension of space $d$, and the shape regularity and quasi-uniform constants of the virtual triangulation $\mathcal T^*_h$.
Consequently,
\begin{equation}\label{eq:gradVknormequivalence}  
\|Q_{K,k-1}^{\div}\nabla v\|_{0,K}\eqsim \|\nabla v\|_{0,K} \quad \forall~v\in V_k(K).
\end{equation}
\end{lemma}
\begin{proof}
Clearly the norm equivalence \eqref{eq:gradVknormequivalence} follows from the local inf-sup condition~\eqref{eq:localdiscreteinfsup}. We will focus on the proof of \eqref{eq:localdiscreteinfsup}.
Without loss of generality, assume $v\in V_k(K)\cap L_0^2(K)$.
Based on DoFs \eqref{Vpdivdof1}-\eqref{Vpdivdof3}, take $\boldsymbol{\phi}\in\mathbb{V}_{k-1}^{\rm div}(K)$ such that
\begin{align*}
(\boldsymbol{\phi}\cdot\boldsymbol{n}, q)_F&=h_K^{-1}(v, q)_F  \quad\quad\,\forall~q\in\mathbb P_{k-1}(F) \textrm{ on each }  F\in\mathcal F(K), \\
(\div\boldsymbol{\phi}, q)_K&=-h_K^{-2}(v, q)_K  \quad\forall~q\in\mathbb P_{\max\{k-2,0\}}(K)/\mathbb R, \\
(\boldsymbol{\phi}, \boldsymbol{q})_K&=0  \qquad\qquad\qquad\forall~\boldsymbol{q}\in \div\skw\mathring{\boldsymbol{V}}_{k}^{d-2}(K). 
\end{align*}
Then $(\boldsymbol{\phi}\cdot\boldsymbol{n})|_{F}=h_K^{-1}Q_{k-1}^Fv$ for $F\in\mathcal F(K)$. Since $\div\boldsymbol{\phi}\in \mathbb P_{\max\{k-2,0\}}(K)$, we have $\div\boldsymbol{\phi}-Q_{0}^K(\div\boldsymbol{\phi})=-h_K^{-2}Q_{k-2}^Kv$. Apply the integration by parts and the fact $v=v-Q_{0}^Kv\in L_0^2(K)$ to get
\begin{align*}
(\boldsymbol{\phi}, \nabla v)_K&=-(\div\boldsymbol{\phi}, v)_K + (\boldsymbol{\phi}\cdot\boldsymbol{n}, v)_{\partial K} \\
&=-(\div\boldsymbol{\phi}-Q_{0}^K(\div\boldsymbol{\phi}), v)_K+\sum_{F\in\mathcal F(K)}(\boldsymbol{\phi}\cdot\boldsymbol{n}, Q_{k-1}^Fv)_{F} \\
&=h_K^{-2}\|Q_{k-2}^Kv\|_{0,K}^2+\sum_{F\in\mathcal F(K)}h_K^{-1}\|Q_{k-1}^Fv\|_{0,F}^2.
\end{align*}
By the norm equivalence \eqref{eq:Vknormequivalence}, we get
\begin{equation}\label{eq:20220204}
\|\nabla v\|_{0,K}^{2}\lesssim h_K^{-2}\|Q_{k-2}^Kv\|_{0,K}^2+\sum_{F\in\mathcal F(K)}h_K^{-1}\|Q_{k-1}^Fv\|_{0,F}^2= (\boldsymbol{\phi}, \nabla v)_K.
\end{equation}

On the other hand, it follows from the integration by parts that
\begin{align*}
\|Q_0^K(\div\boldsymbol{\phi})\|_{0,K}&\lesssim h_{K}^{d/2}\big|Q_0^K(\div\boldsymbol{\phi})\big|\lesssim h_{K}^{-d/2}\big|(\div\boldsymbol{\phi},1)_{K}\big|=h_{K}^{-d/2}\big|(\boldsymbol{\phi}\cdot\boldsymbol{n},1)_{\partial K}\big| \\
&\lesssim \sum_{F\in\mathcal F(K)}h_F^{-1/2}\|\boldsymbol{\phi}\cdot\boldsymbol{n}\|_{0,F}.
\end{align*}
Employing the norm equivalence \eqref{eq:Vkm1divnormequiv},
we acquire
\begin{align*}
\|\boldsymbol{\phi}\|_{0,K}&\eqsim h_K\|\div\boldsymbol{\phi}\|_{0,K} +\sum_{F\in\mathcal F(K)}h_F^{1/2}\|\boldsymbol{\phi}\cdot\boldsymbol{n}\|_{0,F} \\
&\lesssim h_K\|\div\boldsymbol{\phi}-Q_0^K(\div\boldsymbol{\phi})\|_{0,K} +\sum_{F\in\mathcal F(K)}h_F^{1/2}\|\boldsymbol{\phi}\cdot\boldsymbol{n}\|_{0,F}.
\end{align*}
Noting that $\div\boldsymbol{\phi}-Q_{0}^K(\div\boldsymbol{\phi})=-h_K^{-2}Q_{k-2}^Kv$ and $(\boldsymbol{\phi}\cdot\boldsymbol{n})|_{F}=h_K^{-1}Q_{k-1}^Fv$ for $F\in\mathcal F(K)$, we have
\[
\|\boldsymbol{\phi}\|_{0,K}\lesssim h_K^{-1}\|Q_{k-2}^Kv\|_{0,K} +\sum_{F\in\mathcal F(K)}h_F^{-1/2}\|Q_{k-1}^Fv\|_{0,F}.
\]
Then we obtain from the norm equivalence \eqref{eq:Vknormequivalence} and the Poincar\'e-Friedrichs  inequality \cite[(2.14)]{BrennerSung2018} that
\[
\|\boldsymbol{\phi}\|_{0,K}\lesssim h_K^{-1}\|v\|_{0,K}\lesssim \|\nabla v\|_{0,K}.    
\]
Finally, we conclude \eqref{eq:localdiscreteinfsup} from \eqref{eq:20220204} and the last inequality.
\end{proof}

\subsection{Discrete method}
Define the global nonconforming virtual element space
\begin{align*}
V_h&:=\{v_h\in L^2(\Omega): v_h|_K\in V_k(K) \textrm{ for each } K\in\mathcal T_h, \textrm{  } \\
&\qquad\; \textrm{DoFs \eqref{eq:vemdof1} are single-valued for $F\in\mathcal F_h$, and vanish for $F\in \mathcal F_h^{\partial}$}\}.
\end{align*}
We have the discrete Poincar\'e inequality \cite{Brenner2003}
\begin{equation}\label{eq:vempoincareineqlty}
\|v_h\|_0\leq C_p |v_h|_{1,h} \quad\forall~v_h\in V_h,
\end{equation}
where the constant $C_p$ is independent of the mesh size.
Hence $|\cdot|_{1,h}$ is indeed a norm for $V_h$.

Based on the weak formulation \eqref{eq:ellipitc2ndproblemweakform}, we propose a virtual element method without extrinsic stabilization for problem \eqref{eq:ellipitc2ndproblem} as follows: find $u_h\in V_h$ such that
\begin{equation}\label{eq:vem}
a_h(u_h, v_h)=(f, Q_hv_h)\quad\forall~v_h\in V_h,
\end{equation}
where the discrete bilinear form 
\[
a_h(u_h, v_h):=(Q_{h,k-1}^{\div}\nabla_h u_h, Q_{h,k-1}^{\div}\nabla_h v_h)+\alpha(Q_hu_h, Q_hv_h).
\]

\begin{remark}\rm
By introducing $\boldsymbol{\phi}_h=Q_{h,k-1}^{\div}\nabla_h u_h$, the VEM \eqref{eq:vem} can be rewritten as the following primal mixed VEM: find $\boldsymbol{\phi}_h\in\mathbb{V}_{h,k-1}^{\rm div}$ and $u_h\in V_h$ such that
\begin{equation*}
\begin{aligned}
(\boldsymbol{\phi}_h, \boldsymbol{\psi}_h) - (\boldsymbol{\psi}_h, \nabla_h u_h)&=0 \quad\quad\quad\quad\;\;\forall~\boldsymbol{\psi}_h\in\mathbb{V}_{h,k-1}^{\rm div},\\
(\boldsymbol{\phi}_h, \nabla_h v_h)+\alpha(Q_hu_h, Q_hv_h)&=(f, Q_hv_h)\quad\forall~v_h\in V_h.
\end{aligned}
\end{equation*}
\end{remark}

\begin{remark}\rm
The mixed-order HHO method without extrinsic stabilization for the Poisson equation in~\cite{CicuttinErnLemaire2019} is equivalent to find $u_h\in W_h$ such that
\begin{equation}\label{eq:2023111}    
(\nabla_h u_h, \nabla_h v_h)=\sum_{K\in\mathcal T_h}(f, Q_{k-2}^Kv_h)_K\quad\forall~v_h\in W_h,
\end{equation}
where
\begin{align*}
W_h&:=\{v_h\in L^2(\Omega): v_h|_K\in W_k(K) \textrm{ for each } K\in\mathcal T_h, \textrm{  } \\
&\qquad\; \textrm{DoFs \eqref{eq:vemdof1} are single-valued for $F\in\mathcal F_h$, and vanish for $F\in \mathcal F_h^{\partial}$}\}
\end{align*}  
with $W_k(K):=\{v\in H^1(K) : \Delta v\in\mathbb P_{k-2}(K),\, \partial_nv|_F\in\mathbb P_{k-1}(F)\textrm{ for }F\in\mathcal F(K)\}$.
However, $(\nabla_h u_h, \nabla_h v_h)$ is not computable, 
and the method \eqref{eq:2023111} is not the standard nonconforming VEM in \cite{AyusodeDiosLipnikovManzini2016}. For the practical computation, a basis of $W_k(K)$ has to be solved approximately as shown in \cite[Remark 4.1]{CicuttinErnLemaire2019}, then the approximation of the space $W_h$ is no more a virtual element space. 
\end{remark}

It follows from the discrete Poincar\'e inequality \eqref{eq:vempoincareineqlty} that 
\begin{equation}\label{eq:ahbounedness}
a_h(u_h, v_h)\lesssim |u_h|_{1,h}|v_h|_{1,h}\quad\forall~u_h,v_h\in H_0^1(\Omega)+V_h.
\end{equation}  
\begin{lemma}
It holds the coercivity
\begin{equation}\label{eq:ahcoercivity}
|v_h|_{1,h}^2\leq C_e^2 a_h(v_h, v_h)\quad\forall~v_h\in V_h.
\end{equation}  
\end{lemma}
\begin{proof}
Due to \eqref{eq:gradVknormequivalence}, we have
\[
\sum_{K\in\mathcal T_h}\|\nabla_h v_h\|_{0,K}^2\leq C_e^2 \sum_{K\in\mathcal T_h}\|Q_{K,k-1}^{\div}\nabla v_h\|_{0,K}^2\leq C_e^2a_h(v_h, v_h)\quad\forall~v_h\in V_h,
\]
which implies the coercivity \eqref{eq:ahcoercivity}.
\end{proof}

\begin{theorem}
The VEM \eqref{eq:vem} is well-posed.
\end{theorem}
\begin{proof}
Thanks to the boundedness \eqref{eq:ahbounedness} and the coercivity \eqref{eq:ahcoercivity}, we conclude the result from the Lax-Milgram lemma \cite{LaxMilgram1954}.
\end{proof}

\subsection{Error analysis}
\begin{theorem}\label{thm:errorestimateH1}
Let $u\in H_0^1(\Omega)$ be the solution of problem \eqref{eq:ellipitc2ndproblem}, and $u_h\in V_h$ be the solution of the VEM \eqref{eq:vem}. It holds that
\begin{align}
\label{eq:errorestimateH1abs}
|u-u_h|_{1,h}&\leq C_e^2(\|\nabla u-Q_h^{k-1}\nabla u\|_0+\alpha C_p\|u-Q_hu\|_0) \\
\notag
&\quad + \inf_{v_h\in V_h}\big((C_e^2+1)|u-v_h|_{1,h}+\alpha C_p\|u-v_h\|_0\big)\\
\notag
&\quad + C_e^2\sup_{w_h\in V_h}\frac{a(u,w_h)-(f, Q_hw_h)}{|w_h|_{1,h}}.
\end{align}
Assume $u\in H^{k+1}(\Omega)$ and $f\in H^{k-1}(\Omega)$. Then
\begin{equation}\label{eq:errorestimateH1}
|u-u_h|_{1,h}\lesssim h^k(|u|_{k+1}+|f|_{k-1}).
\end{equation}
\end{theorem}
\begin{proof}
Take any $v_h\in V_h$. 
By the definitions of $a_h(\cdot, \cdot)$ and $a(\cdot, \cdot)$, it follows from the discrete Poincar\'e inequality \eqref{eq:vempoincareineqlty} that
\begin{align*}
&\quad a_h(v_h, v_h-u_h)-a(u,v_h-u_h)\\
&=(Q_{h,k-1}^{\div}\nabla_h v_h, Q_{h,k-1}^{\div}\nabla_h(v_h-u_h))-(\nabla u, \nabla_h(v_h-u_h))\\
&\quad +\alpha(Q_hv_h, Q_h(v_h-u_h))-\alpha(u, v_h-u_h) \\
&=(Q_{h,k-1}^{\div}\nabla_h v_h-\nabla u, \nabla_h(v_h-u_h))+\alpha(Q_hv_h-u,v_h-u_h)\\
&\leq (\|\nabla u-Q_{h,k-1}^{\div}\nabla_h v_h\|_0+\alpha C_p\|u-Q_hv_h\|_0)|v_h-u_h|_{1,h}.
\end{align*}
Apply the coercivity \eqref{eq:ahcoercivity} and \eqref{eq:vem} to get
\begin{equation*}
C_e^{-2}|v_h-u_h|_{1,h}^2\leq a_h(v_h-u_h, v_h-u_h)=a_h(v_h, v_h-u_h)-(f, Q_h(v_h-u_h)).
\end{equation*}
Combining the last two inequalities yields
\begin{align*}
|v_h-u_h|_{1,h}&\leq C_e^2(\|\nabla u-Q_{h,k-1}^{\div}\nabla_h v_h\|_0+\alpha C_p\|u-Q_hv_h\|_0) \\
&\quad + C_e^2\sup_{w_h\in V_h}\frac{a(u,w_h)-(f, Q_hw_h)}{|w_h|_{1,h}}.
\end{align*}
Since $\mathbb P_{k-1}(K;\mathbb R^d)\subseteq\mathbb{V}_{k-1}^{\rm div}(K)$ for $K\in\mathcal T_h$, we have $Q_{h,k-1}^{\div}(Q_h^{k-1}\nabla u)=Q_h^{k-1}\nabla u$. Hence
\begin{align*}    
\|\nabla u-Q_{h,k-1}^{\div}\nabla_h v_h\|_0&\leq \|\nabla u-Q_{h,k-1}^{\div}\nabla u\|_0+\|Q_{h,k-1}^{\div}(\nabla u-\nabla_h v_h)\|_0\\
&= \|\nabla u-Q_h^{k-1}\nabla u-Q_{h,k-1}^{\div}(\nabla u-Q_h^{k-1}\nabla u)\|_0\\
&\quad+\|Q_{h,k-1}^{\div}(\nabla u-\nabla_h v_h)\|_0 \\
&\leq \|\nabla u-Q_h^{k-1}\nabla u\|_0+|u-v_h|_{1,h}.
\end{align*}
Similarly, we have
\[
\|u-Q_hv_h\|_0\leq \|u-Q_hu\|_0+\|Q_h(u-v_h)\|_0\leq \|u-Q_hu\|_0+\|u-v_h\|_0.
\]
Then we obtain from the last three inequalities that
\begin{align*}
|v_h-u_h|_{1,h}&\leq C_e^2\big(\|\nabla u-Q_h^{k-1}\nabla u\|_0+|u-v_h|_{1,h}+\alpha C_p(\|u-Q_hu\|_0+\|u-v_h\|_0)\big) \\
&\quad + C_e^2\sup_{w_h\in V_h}\frac{a(u,w_h)-(f, Q_hw_h)}{|w_h|_{1,h}}.
\end{align*}
Thus, we acquire \eqref{eq:errorestimateH1abs} from the last inequality and the triangle inequality.

Next we derive estimate \eqref{eq:errorestimateH1}. Recall the consistency error estimate in \cite[Lemma 5.5]{ChenHuang2020ncvem}
\[
a(u, w_h)+(f, w_h)\lesssim h^k|u|_{k+1}|w_h|_{1,h}\quad\forall~w_h\in V_h.
\]
Then 
\begin{align*}
a(u, w_h)-(f, Q_hw_h) 
&=a(u, w_h)+(f, w_h)+(f-Q_hf, w_h) \\
&=a(u, w_h)+(f, w_h)+(f-Q_hf, w_h-Q_h^0w_h) \\
&\lesssim h^k(|u|_{k+1}+|f|_{k-1})|w_h|_{1,h}.
\end{align*}
At last, \eqref{eq:errorestimateH1} follows from \eqref{eq:errorestimateH1abs} and the approximation of $V_h$ \cite{ChenHuang2020ncvem}.
\end{proof}

\section{Conforming virtual element method without extrinsic stabilization}\label{sec:stabfreecfmvem}

In this section we will develop a conforming VEM without extrinsic stabilization for the second order elliptic problem \eqref{eq:ellipitc2ndproblem}.

We additionally impose the following condition on the mesh $\mathcal T_h$ in this section:
\begin{itemize}
 \item[(A3)] All the $(d-1)$-dimensional faces of $K\in \mathcal T_h$ are simplices.
\end{itemize}




\subsection{$H^1$-conforming virtual element}

Recall the $H^1$-conforming virtual element in \cite{ChenHuangWei2022,AhmadAlsaediBrezziMariniEtAl2013,BeiraoBrezziCangianiManziniEtAl2013,BeiraoBrezziMariniRusso2014}.
By assumption (A3), we can require the shape function restricted to each face to be a polynomial when defining an $H^1$-conforming virtual element.
To this end, let the space of shape functions be
\begin{align*}
V_k(K)&:=\big\{ v\in H^1(K): \Delta v\in \mathbb P_{k}(K),  v|_{\partial K}\in H^1(\partial K),  v|_F\in \mathbb P_{k}(F)\;\;\forall~F\in\mathcal F(K),\\    
&\qquad\qquad\qquad\qquad\qquad\qquad\qquad\;\;\; \textrm{and } (v-\Pi_k^Kv, q)_K=0\quad\forall~q\in\mathbb P_{k-2}^{\perp}(K)\},
\end{align*}
where $\Pi_k^K$ is defined by \eqref{eq:projlocal2d1}-\eqref{eq:projlocal2d2}.
It holds $\mathbb P_k(K)\subseteq V_k(K)$.
The degrees of freedom are given by
\begin{align}
v(\delta), & \quad \delta\in\Delta_0(K), \label{eq:cfmvemdof1}\\
\frac{1}{|e|}(v, \phi_i^e)_e, & \quad e\in\Delta_j(K), i=1,\ldots, \dim\mathbb P_{k-j-1}(e), j=1,\ldots, d-1, \label{eq:cfmvemdof2}\\
\frac{1}{|K|}(v, \phi_i^K)_K, & \quad i=1,\ldots, \dim\mathbb P_{k-2}(K), \label{eq:cfmvemdof3}
\end{align}
where $\{\phi_i^e\}_{i=1}^{\dim\mathbb P_{k-j-1}(e)}$ is a basis of space $\mathbb P_{k-j-1}(e)$, and $\{\phi_i^K\}_{i=1}^{\dim\mathbb P_{k-2}(K)}$ a basis of space $\mathbb P_{k-2}(K)$.

For $v\in V_k(K)$, the $H^1$ projection $\Pi_k^Kv$ and the $L^2$ projection $Q_k^Kv= \Pi_k^Kv + Q_{k-2}^Kv-Q_{k-2}^K\Pi_k^Kv$ are computable using the DoFs \eqref{eq:cfmvemdof1}-\eqref{eq:cfmvemdof3}. 
Following the argument in \cite[Lemma 4.7]{ChenHuangWei2022} and \cite{ChenHuang2018,BrennerSung2018,BeiraodaVeigaLovadinaRusso2017},
we have the norm equivalence of space $V_k(K)$, that is for $v\in V_k(K)$, it holds
\begin{equation}\label{eq:cfmVknormequivalence}
h_K^2|v|_{1,K}^2\lesssim\|v\|_{0,K}^2\eqsim \|Q_{k-2}^Kv\|_{0,K}^2 + h_K\|v\|_{0,\partial K}^2.
\end{equation}

Employing the same argument as in Lemma~\ref{lem:gradVknormequivalence}, from  \eqref{eq:cfmVknormequivalence}, we get the norm equivalence
\begin{equation*}
\|Q_{K,k}^{\div}\nabla v\|_{0,K}\eqsim \|\nabla v\|_{0,K} \quad \forall~v\in V_k(K).
\end{equation*}

\begin{remark}\rm
When $k\geq2$, we can replace $Q_{K,k}^{\div}$ by the $L^2$-orthogonal projection operator onto space $\mathbb{V}_{k,k-2}^{\rm div}(K)$, where
\begin{align*}
\mathbb{V}_{k,k-2}^{\rm div}(K):=&\{\boldsymbol{\phi}\in\mathbb{V}_{k}^{\rm div}(K): \div\boldsymbol{\phi}\in\mathbb P_{k-2}(K)\} \\
=&\{\boldsymbol{\phi}\in\boldsymbol{V}_{k}^{\mathrm{BDM}}(K): \div\boldsymbol{\phi}\in\mathbb P_{k-2}(K), \boldsymbol{\phi}\cdot\boldsymbol{n}|_F\in\mathbb P_{k}(F)\;\;\forall~F\in\mathcal F(K)\}.
\end{align*}
\end{remark}

\subsection{Discrete method}
Define the global conforming virtual element space
\[
V_h:=\{v_h\in H_0^1(\Omega): v_h|_K\in V_k(K) \textrm{ for each } K\in\mathcal T_h\}.
\]

Based on the weak formulation \eqref{eq:ellipitc2ndproblemweakform}, we propose a virtual element method without extrinsic stabilization for problem \eqref{eq:ellipitc2ndproblem} as follows: find $u_h\in V_h$ such that
\begin{equation}\label{eq:cfmvem}
a_h(u_h, v_h)=(f, Q_hv_h)\quad\forall~v_h\in V_h,
\end{equation} 
where the discrete bilinear form 
\[
a_h(u_h, v_h):=(Q_{h,k}^{\div}\nabla u_h, Q_{h,k}^{\div}\nabla v_h)+\alpha(Q_hu_h, Q_hv_h).
\]
The VEM \eqref{eq:cfmvem} is uni-solvent.


By introducing $\boldsymbol{\phi}_h=Q_{h,k}^{\div}\nabla u_h$, the VEM \eqref{eq:cfmvem} can be rewritten as the following primal mixed VEM: find $\boldsymbol{\phi}_h\in\mathbb{V}_{h,k}^{\rm div}$ and $u_h\in V_h$ such that
\begin{equation*}
\begin{aligned}
(\boldsymbol{\phi}_h, \boldsymbol{\psi}_h) - (\boldsymbol{\psi}_h, \nabla u_h)&=0 \quad\quad\quad\quad\;\;\forall~\boldsymbol{\psi}_h\in\mathbb{V}_{h,k}^{\rm div},\\
(\boldsymbol{\phi}_h, \nabla v_h)+\alpha(Q_hu_h, Q_hv_h)&=(f, Q_hv_h)\quad\forall~v_h\in V_h.
\end{aligned} 
\end{equation*} 



Applying the standard error analysis for VEMs, we have the following error estimate for the VEM \eqref{eq:cfmvem}.
\begin{theorem}\label{thm:cfmerrorestimateH1}
Let $u\in H_0^1(\Omega)$ be the solution of problem \eqref{eq:ellipitc2ndproblem}, and $u_h\in V_h$ be the solution of the VEM \eqref{eq:cfmvem}. Assume $u\in H^{k+1}(\Omega)$ and $f\in H^{k-1}(\Omega)$. Then
\begin{equation*}
|u-u_h|_1\lesssim h^k(|u|_{k+1}+|f|_{k-1}).
\end{equation*}
\end{theorem}

\section{Numerical results}\label{sec:numericalexamps}
In this section, we will numerically test the nonconforming virtual element method \eqref{eq:vem} and the conforming
virtual element method \eqref{eq:cfmvem}, which are abbreviated as SFNCVEM and SFCVEM respectively.
For the convenience of narration, we also abbreviate the standard conforming virtual element
method in \cite{BeiraoBrezziCangianiManziniEtAl2013} and non-conforming virtual element method  in \cite{CangianiManziniSutton2017} as CVEM and NCVEM respectively. 
We implement all the experiments by using the FEALPy package \cite{fealpy} on a PC with AMD Ryzen
5 3500U CPU and 64-bit Ubuntu 22.04 operating system.
Set the rectangular domain $\Omega = (0, 1)\times(0, 1)$.

\subsection{Verification of convergence}
Consider the second order elliptic problem~\eqref{eq:ellipitc2ndproblem} with $\alpha = 2$.
The exact solution and source term are given by
\[
u = \sin(\pi x)\sin(\pi y), \quad f = (2\pi^2+2)\sin(\pi x)\sin(\pi y).
\]

The rectangular domain $\Omega$ is partitioned by the convex polygon mesh
$\mathcal T_0$ and non-convex polygon mesh $\mathcal T_1$ in Fig. \ref{fig:mesh}, respectively. 
We choose $k = 1, 2, 5$ in both SFNCVEM and SFCVEM.
The numerical results of the SFNCVEM on meshes $\mathcal T_0$ and $\mathcal T_1$ are shown in Fig.~\ref{fig:rate1}. We can see that $\|u - Q_h u_h\|_0=O(h^{k+1})$ and $\|\nabla u - Q_{h, k-1}^{\div}\nabla_h u_h\|_0=O(h^{k})$, which coincide with Theorem~\ref{thm:errorestimateH1}.
And the numerical results of the SFCVEM are presented in Fig.~\ref{fig:rate2}. Again $\|u - Q_h u_h\|_0=O(h^{k+1})$ and $\|\nabla u - Q_{h, k}^{\div}\nabla u_h\|_0=O(h^{k})$, which confirm the
theoretical convergence rate in Theorem~\ref{thm:cfmerrorestimateH1}.
\begin{figure}[htbp]
\subfigure[Convex polygon mesh $\mathcal T_0$.]{
\begin{minipage}[t]{0.425\linewidth}
\centering
\includegraphics*[width=4.7cm]{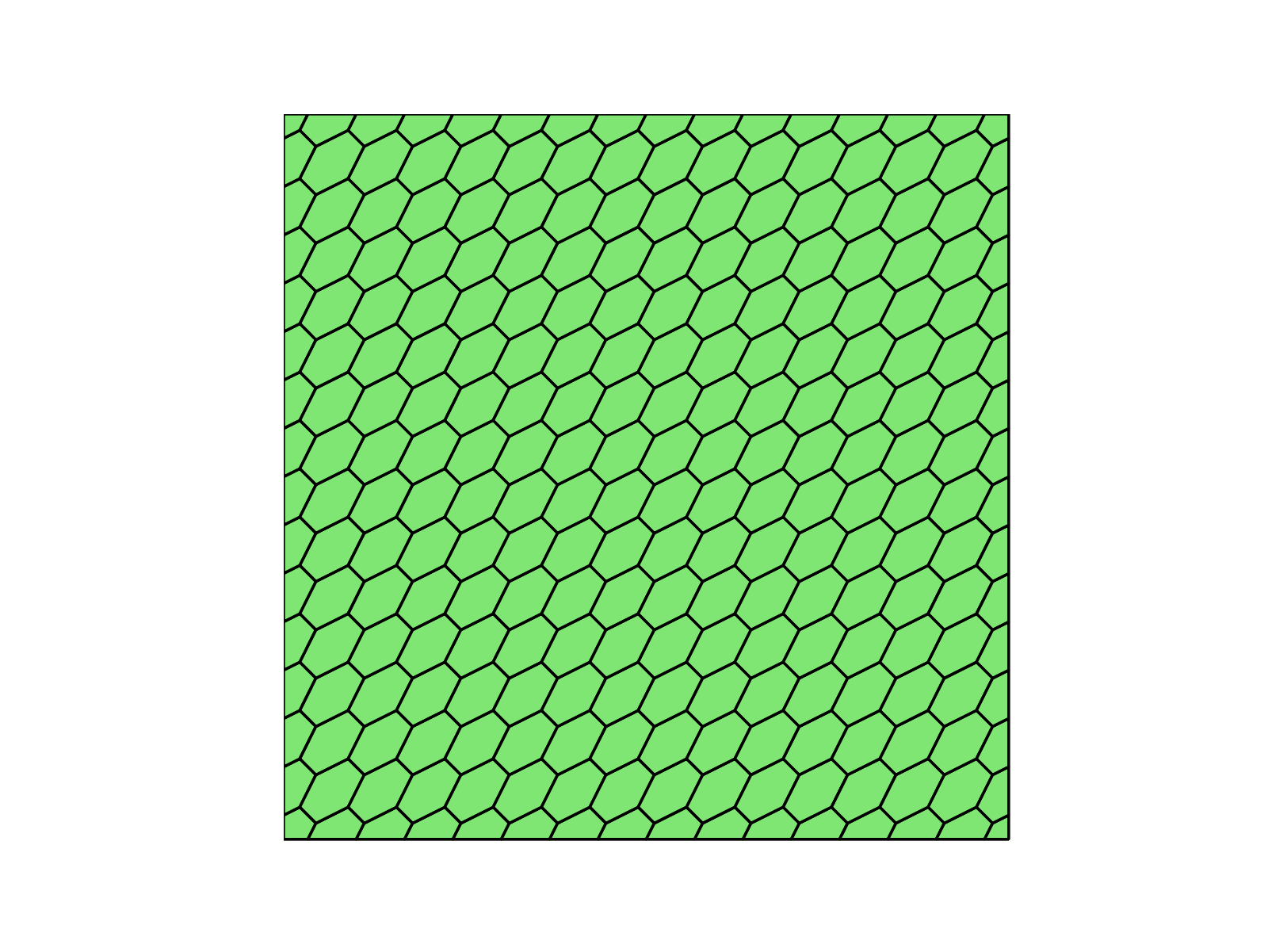}
\end{minipage}}
\quad \quad
\subfigure[Non-convex polygon mesh 
$\mathcal T_1$.]
{\begin{minipage}[t]{0.425\linewidth}
\centering
\includegraphics*[width=4.7cm]{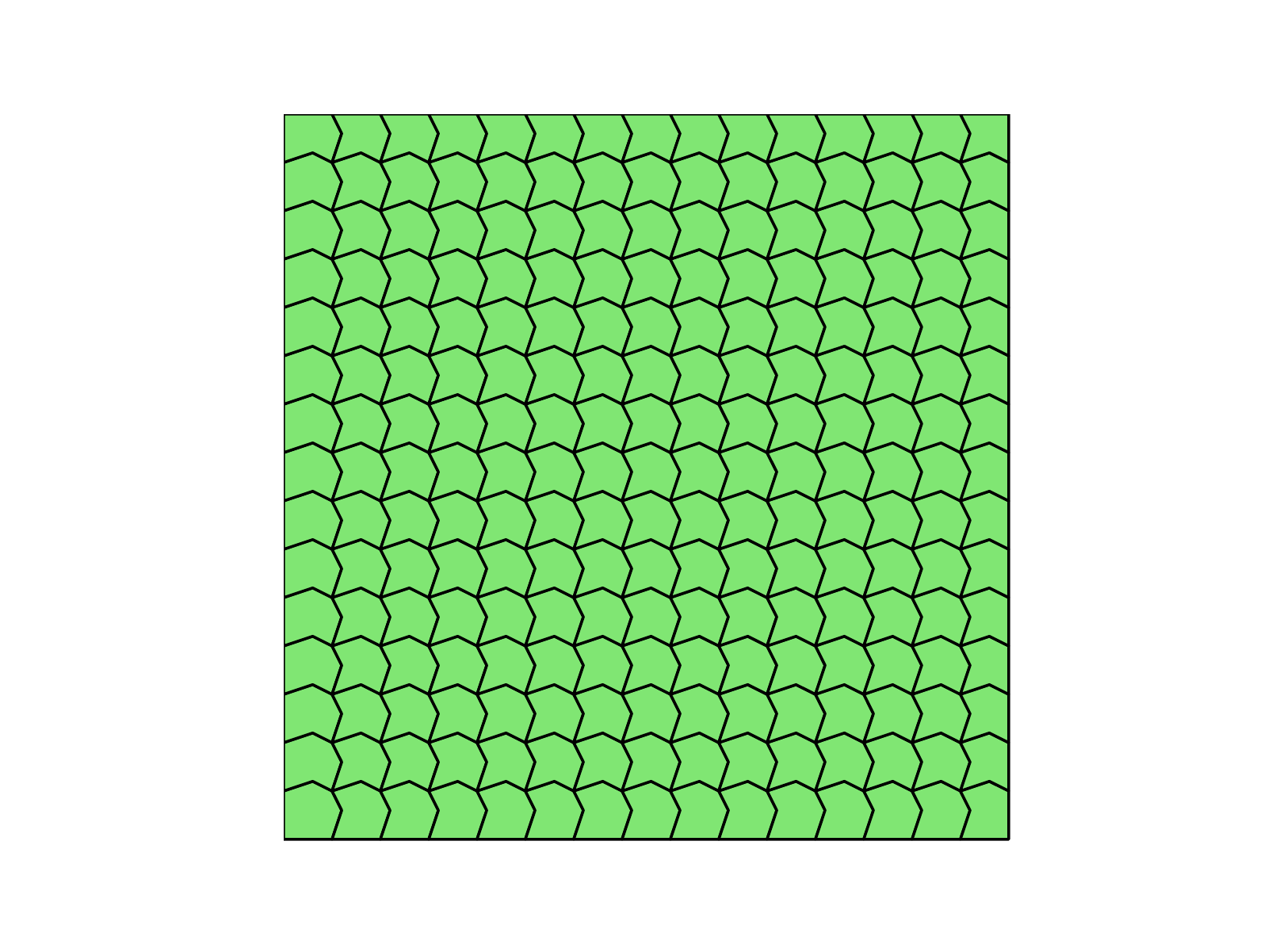}
\end{minipage}}
\caption{Convex polygon mesh and non-convex polygon mesh.}
\label{fig:mesh}
\end{figure}
\begin{figure}[htbp]
\centering
\begin{minipage}[t]{0.49\linewidth}
\centering
\includegraphics[width=5.5cm]{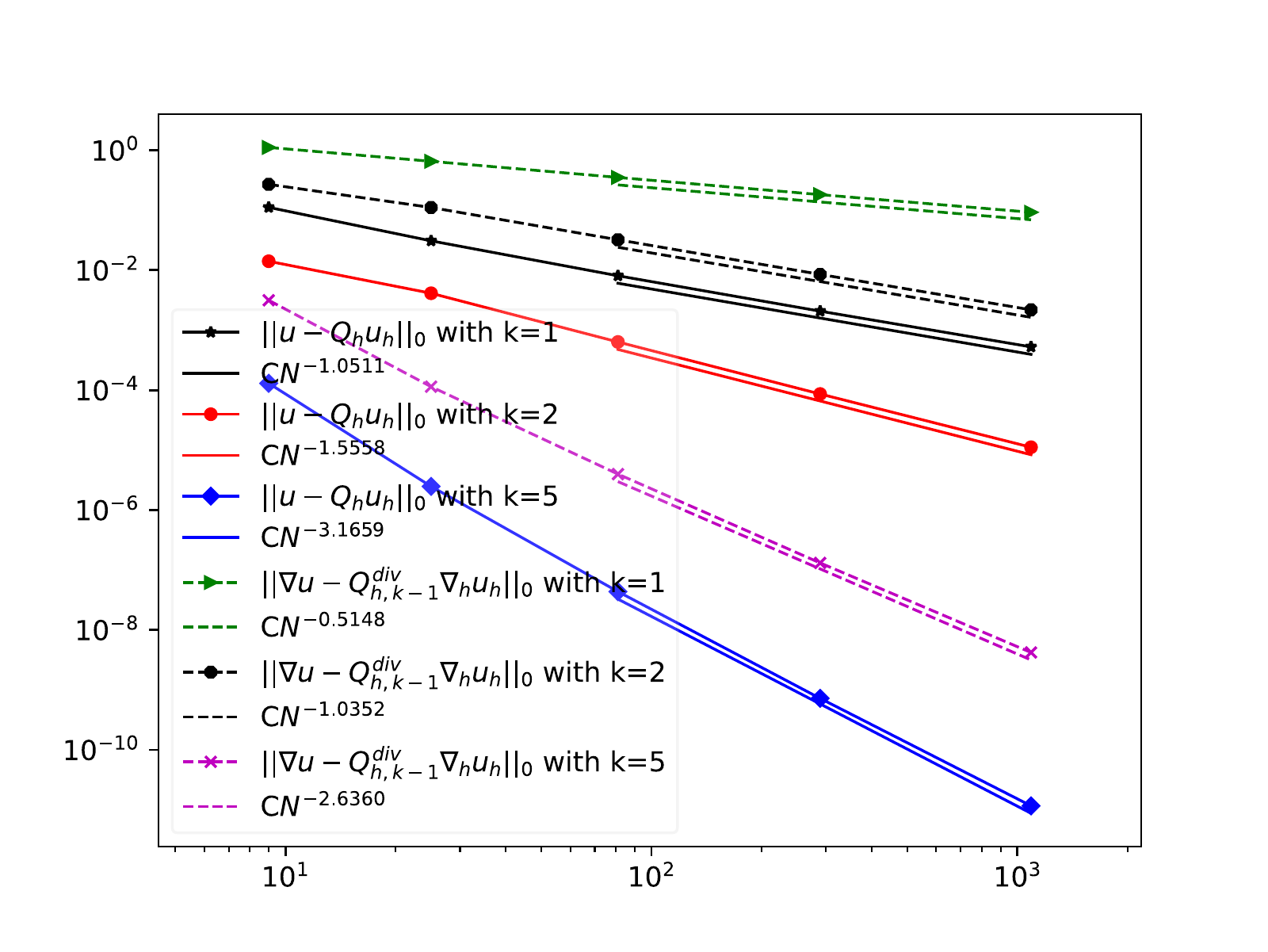}
\end{minipage}%
\begin{minipage}[t]{0.49\linewidth}
\centering
\includegraphics[width=5.5cm]{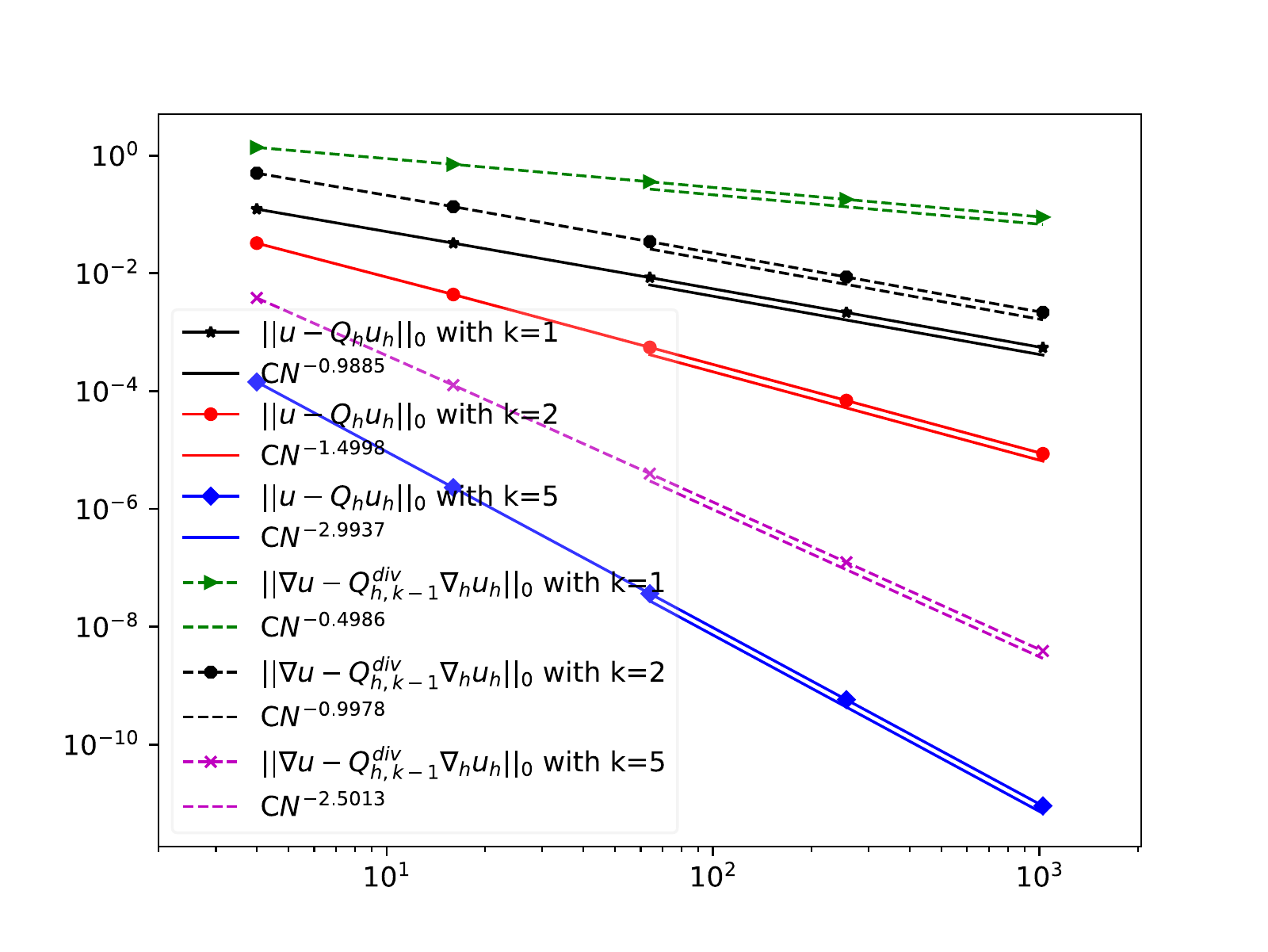}
\end{minipage}%
\centering
\caption{Errors $\|u - Q_h u_h\|_0$ and $\|\nabla u - Q_{h, k-1}^{\div}\nabla_h u_h\|_0$ 
of nonconforming VEM \eqref{eq:vem} on 
$\mathcal T_0$(left) and $\mathcal T_1$(right) with $k=1, 2, 5$.}
\label{fig:rate1}
\end{figure}

\begin{figure}[htbp]
\centering
\begin{minipage}[t]{0.49\linewidth}
\centering
\includegraphics[width=5.5cm]{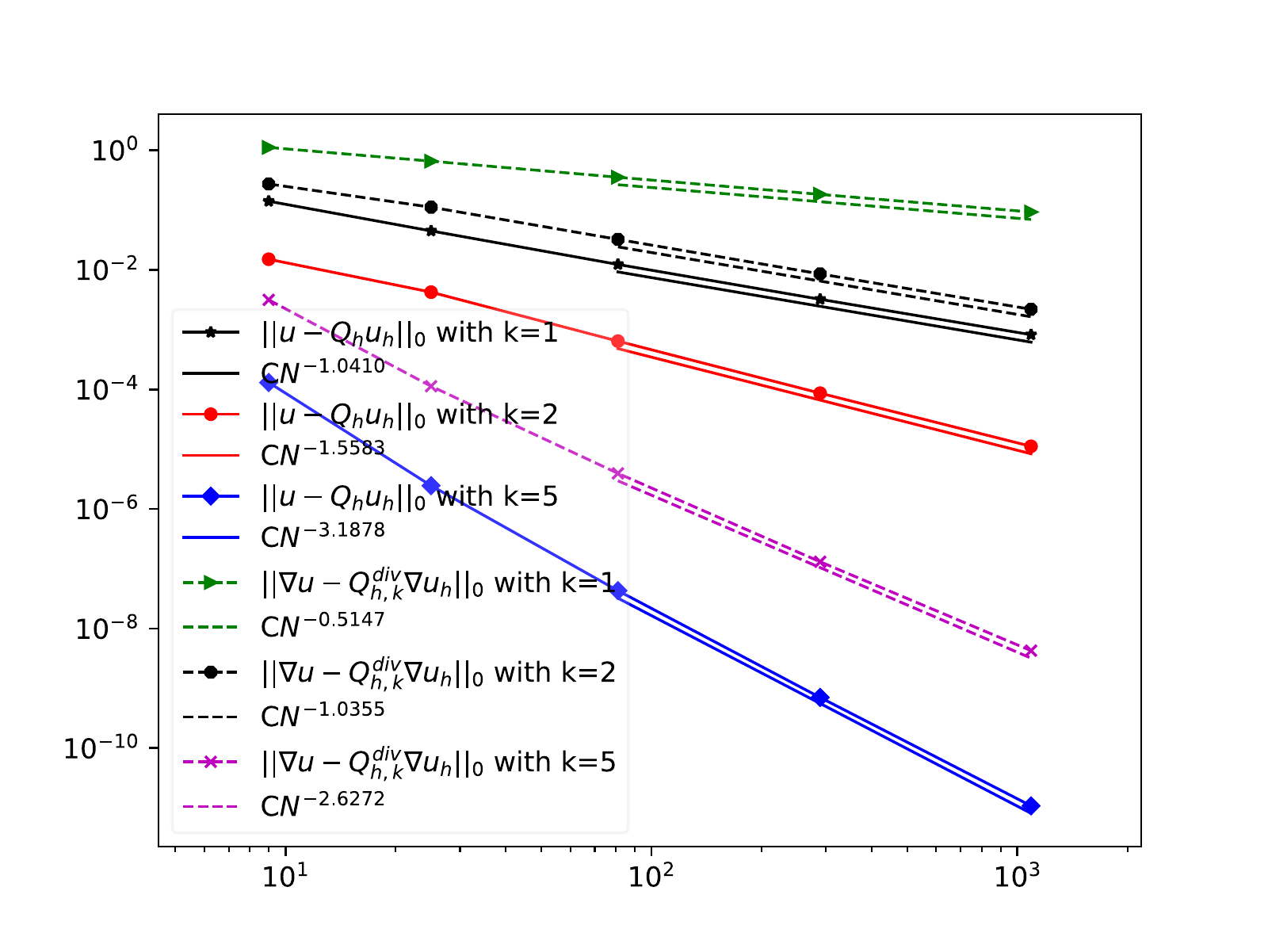}
\end{minipage}%
\begin{minipage}[t]{0.49\linewidth}
\centering
\includegraphics[width=5.5cm]{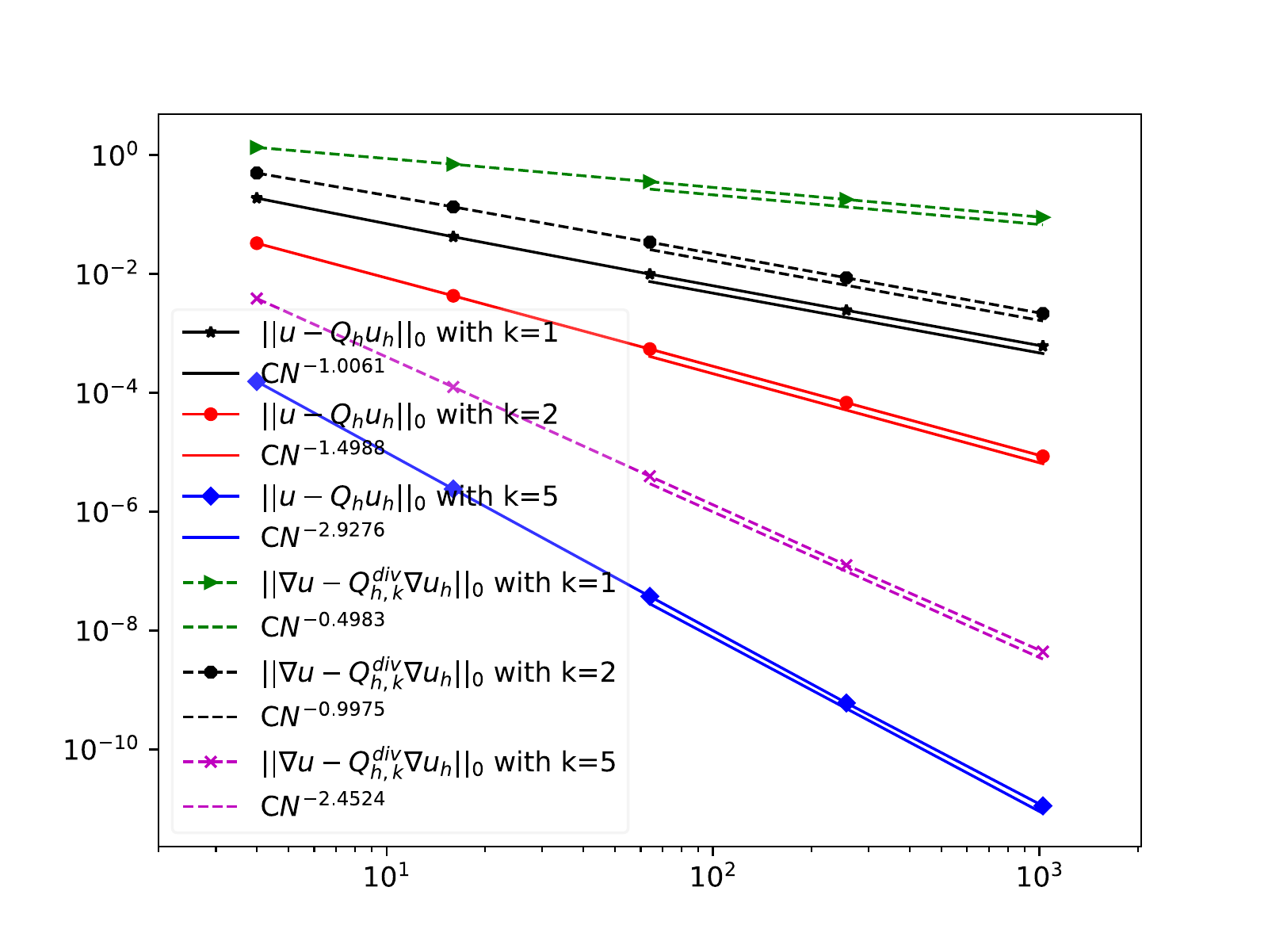}
\end{minipage}%
\centering
\caption{Errors $\|u - Q_h u_h\|_0$ and $\|\nabla u - Q_{h, k}^{\div}\nabla u_h\|_0$ 
of conforming VEM \eqref{eq:cfmvem} on $\mathcal T_0$(left) and 
$\mathcal T_1$(right) with $k=1, 2, 5$.}
\label{fig:rate2}
\end{figure}

\subsection{The invertibility of the local stiffness matrices}

We construct three different hexagons shown in Fig. \ref{fig:hexagon}, and calculate the eigenvalues of local stiffness matrices with $k=3$ for four virtual element methods. Our numerical results show that, on all the three hexagons, both SFNCVEM and SFCVEM have only one zero eigenvalue. In Tables \ref{tab:comparison0}-\ref{tab:comparison2}, we also present the minimum non-zero eigenvalue, the maximum eigenvalue and the condition number for the local stiffness matrix on different hexagons in Fig. \ref{fig:hexagon}, from which we can see that these quantities are comparable for four virtual element methods.
\begin{figure}[htbp]
\subfigure[Regular hexagon.]{
\begin{minipage}[t]{0.3\linewidth}
\centering
\includegraphics*[width=1in]{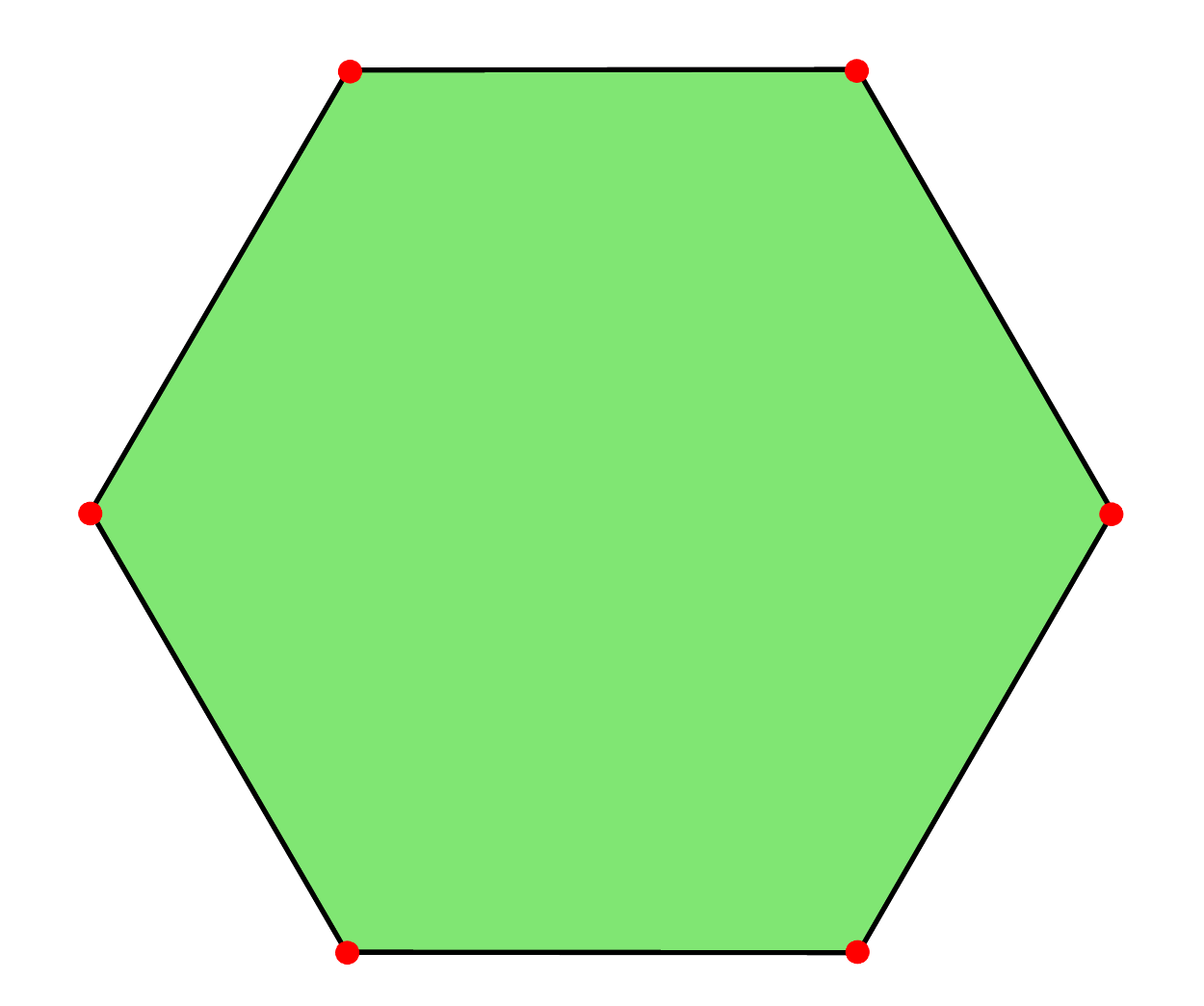}
\end{minipage}}
\;\;
\subfigure[Quasi-regular hexagon generated by regular
    hexagon with a small perturbation.]
{\begin{minipage}[t]{0.3\linewidth}
\centering
\includegraphics*[width=1in]{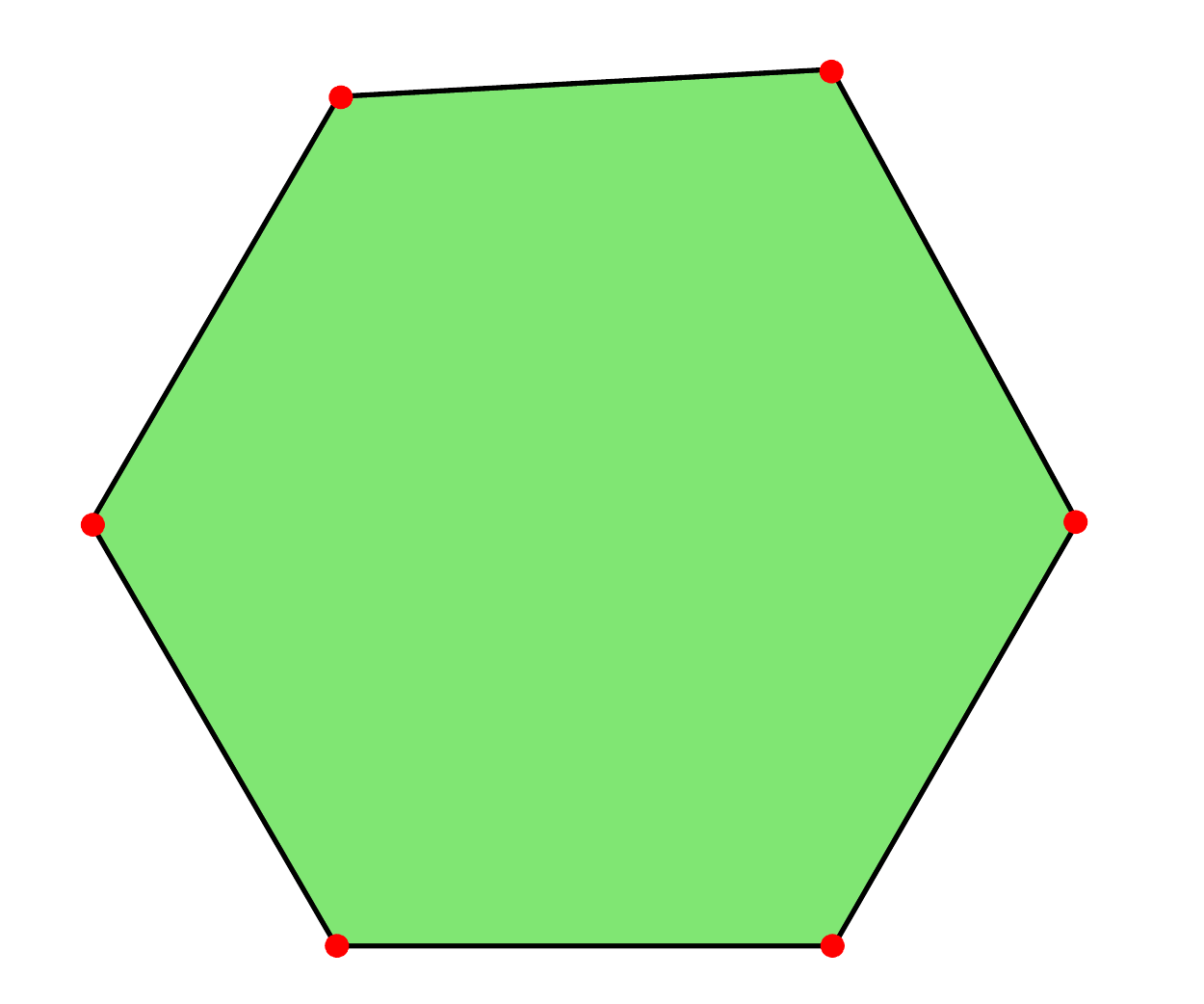}
\end{minipage}}
\;\;
\subfigure[Square with two hanging nodes.]
{\begin{minipage}[t]{0.35\linewidth}
\centering
\includegraphics*[width=0.9in]{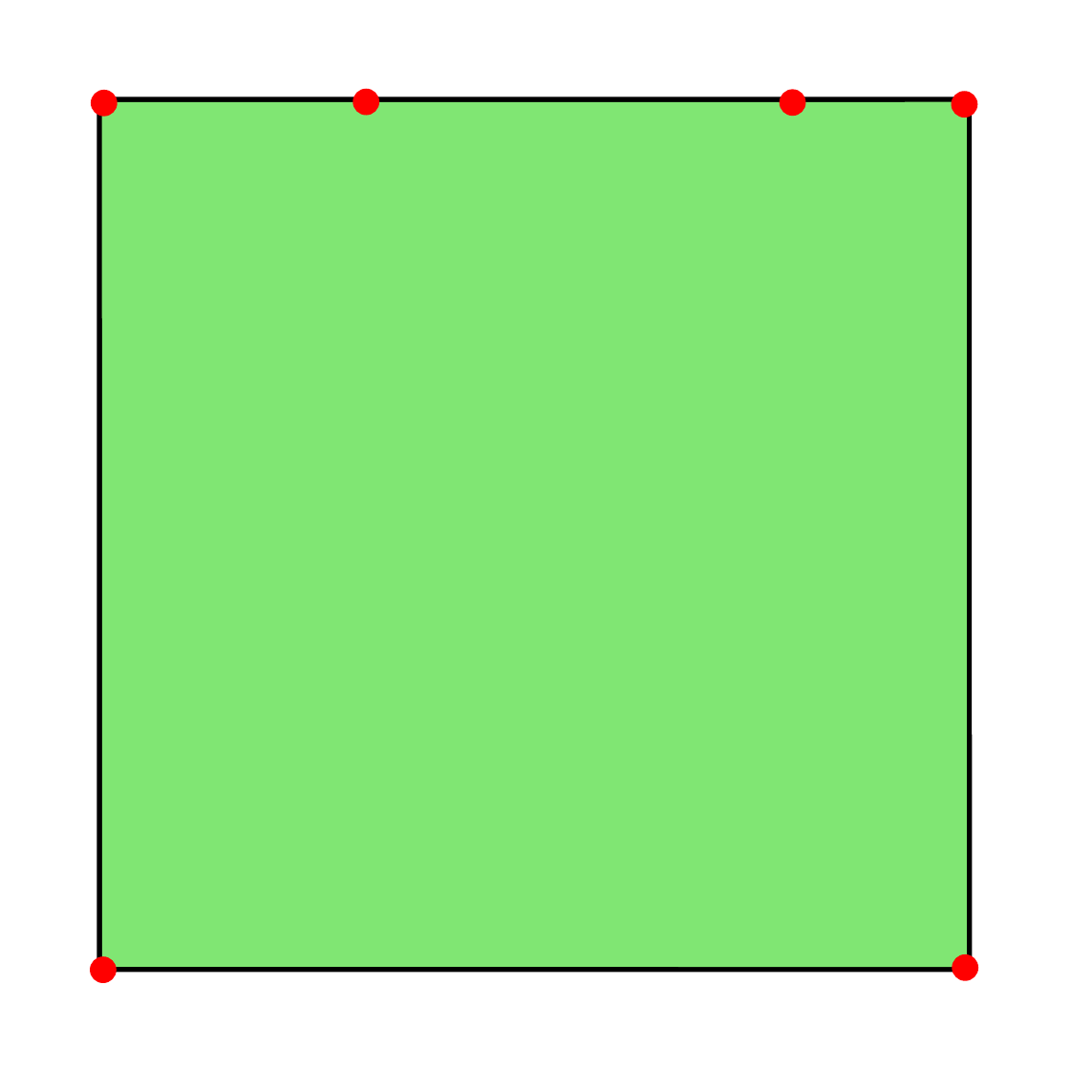}
\end{minipage}}
\caption{The regular hexagon(Left), the quasi-regular hexagon generated by regular
    hexagon with a small perturbation (Middle), 
and the square with two hanging nodes (Right).}
\label{fig:hexagon}
\end{figure}
\begin{table}[htbp]
\centering
\caption{Comparison of eigenvalues and condition numbers on the regular hexagon.}
\label{tab:comparison0}
\begin{tabular}{c|p{3cm}<{\centering}p{3.5cm}<{\centering}p{3cm}<{\centering}}
\hline
\textbf{Method} & \textbf{Maximum eigenvalue} & \textbf{Minimum nonzero
eigenvalue} & \textbf{Condition number} \\ \hline
NCVEM & 975.5693189 & 0.309674737 & 3150.303211 \\ \hline
CVEM & 1012.488116 & 0.297206358 & 3406.683909 \\ \hline
SFNCVEM & 992.5956147 & 0.318932029 & 3112.248147 \\ \hline
SFCVEM & 1011.173331 & 0.298509692 & 3387.405362 \\
\hline
\end{tabular}
\end{table}
\begin{table}[htbp]
\centering
\caption{Comparison of eigenvalues and condition numbers on the quasi-regular hexagon.}
\label{tab:comparison1}
\begin{tabular}{c|p{3cm}<{\centering}p{3.5cm}<{\centering}p{3cm}<{\centering}}
\hline
\textbf{Method} & \textbf{Maximum eigenvalue} & \textbf{Minimum nonzero
eigenvalue} & \textbf{Condition number} \\ \hline
NCVEM   & 935.2883848 & 0.279027715 & 3351.955143 \\ \hline
CVEM    & 1014.672395 & 0.257370621 & 3942.456177 \\ \hline
SFNCVEM & 997.4831245 & 0.282126359 & 3535.589964 \\ \hline
SFCVEM  & 1047.876056 & 0.258970708 &
4046.311124 \\
\hline
\end{tabular}
\end{table}
\begin{table}[htbp]
\centering
\caption{Comparison of eigenvalues and condition numbers on the square with two hanging nodes.}
\label{tab:comparison2}
\begin{tabular}{c|p{3cm}<{\centering}p{3.5cm}<{\centering}p{3cm}<{\centering}}
\hline
\textbf{Method} & \textbf{Maximum eigenvalue} & \textbf{Minimum nonzero
eigenvalue} & \textbf{Condition number} \\ \hline
 NCVEM   & 941.8571938&	0.21069027	&4470.340249 \\ \hline
 CVEM    & 1046.755495&	0.200435123	&5222.4155 \\ \hline
 SFNCVEM & 986.5963357&	0.212761106	&4637.108513 \\ \hline
 SFCVEM  & 1061.651989&	0.202074633	&5253.761808 \\
\hline
\end{tabular}
\end{table}

\subsection{Comparison of assembling time}
The only difference between the standard VEMs and the VEMs without extrinsic stabilization is the
stiffness matrix, so we compare the time consumed in assembling the stiffness matrix of
four different VEMs in detail by varying the degree $k$ and the mesh size $h$
respectively. We use the mesh in Fig.~\ref{fig:mesh}(a) for this experiment.
The results presented in Tables
\ref{tab:ptime} and \ref{tab:ttime} show that NCVEM, CVEM and SFNCVEM have similar assembling time. 
However, SFCVEM requires more time due to the projection onto the one-order higher polynomial space.
\begin{table}[htbp]
\caption{Time consumed in assembling stiffness matrix of four VEMs with
 $h=0.2$ and different $k$.}
\label{tab:ptime}
\centering
\begin{tabular}{c|cccc}
\hline
$k$ & 2 & 4 & 8 & 10 \\
\hline
SFCVEM & 0.053684235 & 0.144996881 & 1.468627453 & 2.603836536 \\
SFNCVEM & 0.022516727 & 0.065697193 & 0.806378841 & 1.554260015 \\
CVEM & 0.021185875 & 0.059809923 & 0.600241184 & 1.160929918 \\
NCVEM & 0.0213027 & 0.061014891 & 0.596506596 & 1.129639149 \\
\hline
\end{tabular}
\end{table}
\begin{table}[htbp]
\caption{Time consumed in assembling stiffness matrix of four VEMs with $k = 5$ and different $h$.}
\label{tab:ttime}
\centering
\begin{tabular}{c|cccc}
\hline
$h$ &    1	& 0.25 & 0.0625 & 0.03125\\
\hline
SFCVEM & 0.039689541 & 0.199015379	& 1.74412179	& 4.75462532\\
SFNCVEM & 0.018287182 & 0.100006819	& 0.81251812	& 2.465409517\\
CVEM & 0.018686771 & 0.087426662	& 0.781031132	& 1.983617783\\
NCVEM & 0.018309593 & 0.096345425	& 0.767129898	& 2.159288645\\
\hline
\end{tabular}
\end{table}

\subsection{Condition number of the stiffness
matrix}
\begin{figure}[htbp]
\centering
\subfigure{\includegraphics[width=1in]{./figures/hexagon0.pdf}}
\subfigure{\includegraphics[width=1in]{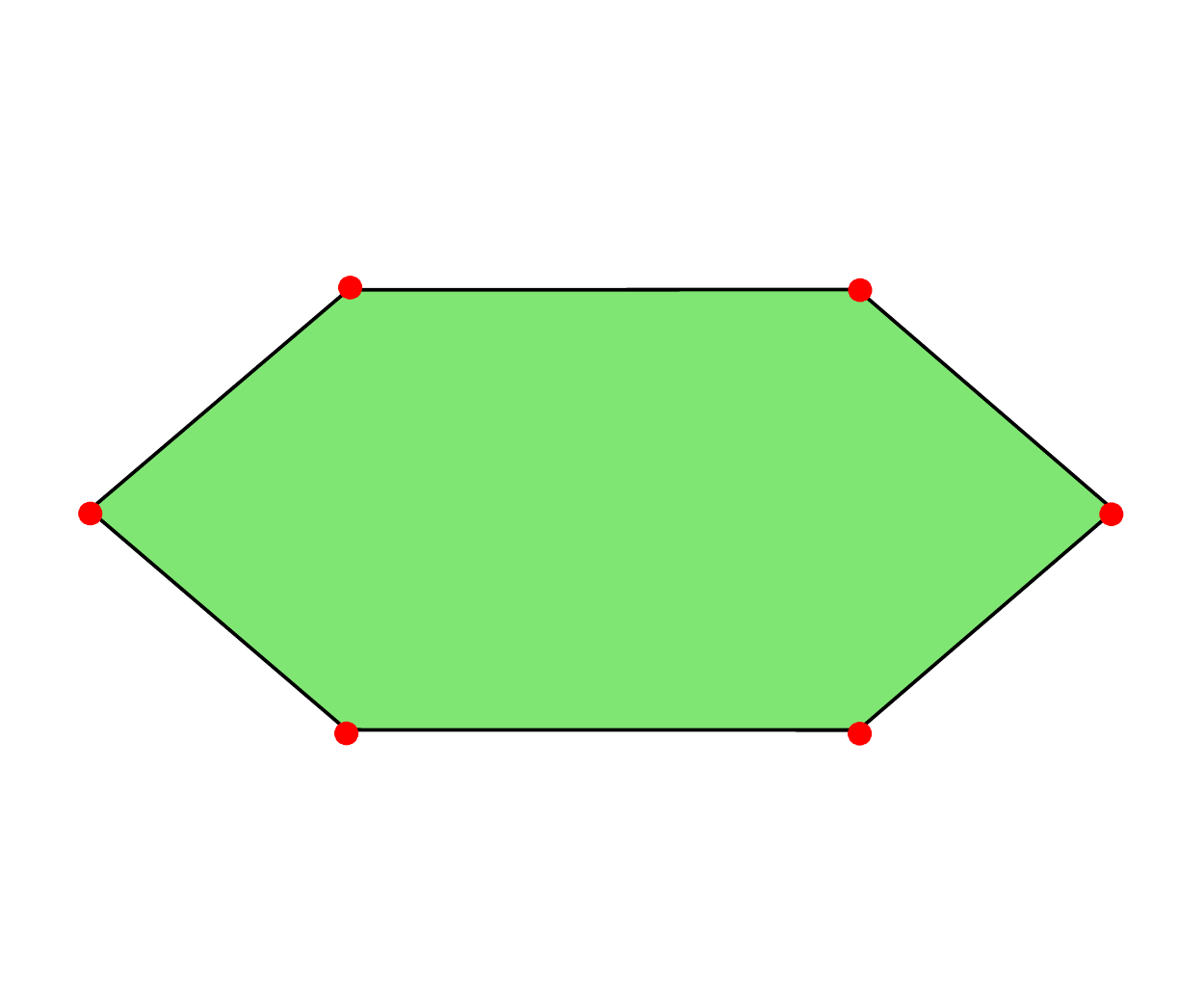}}
\subfigure{\includegraphics[width=1in]{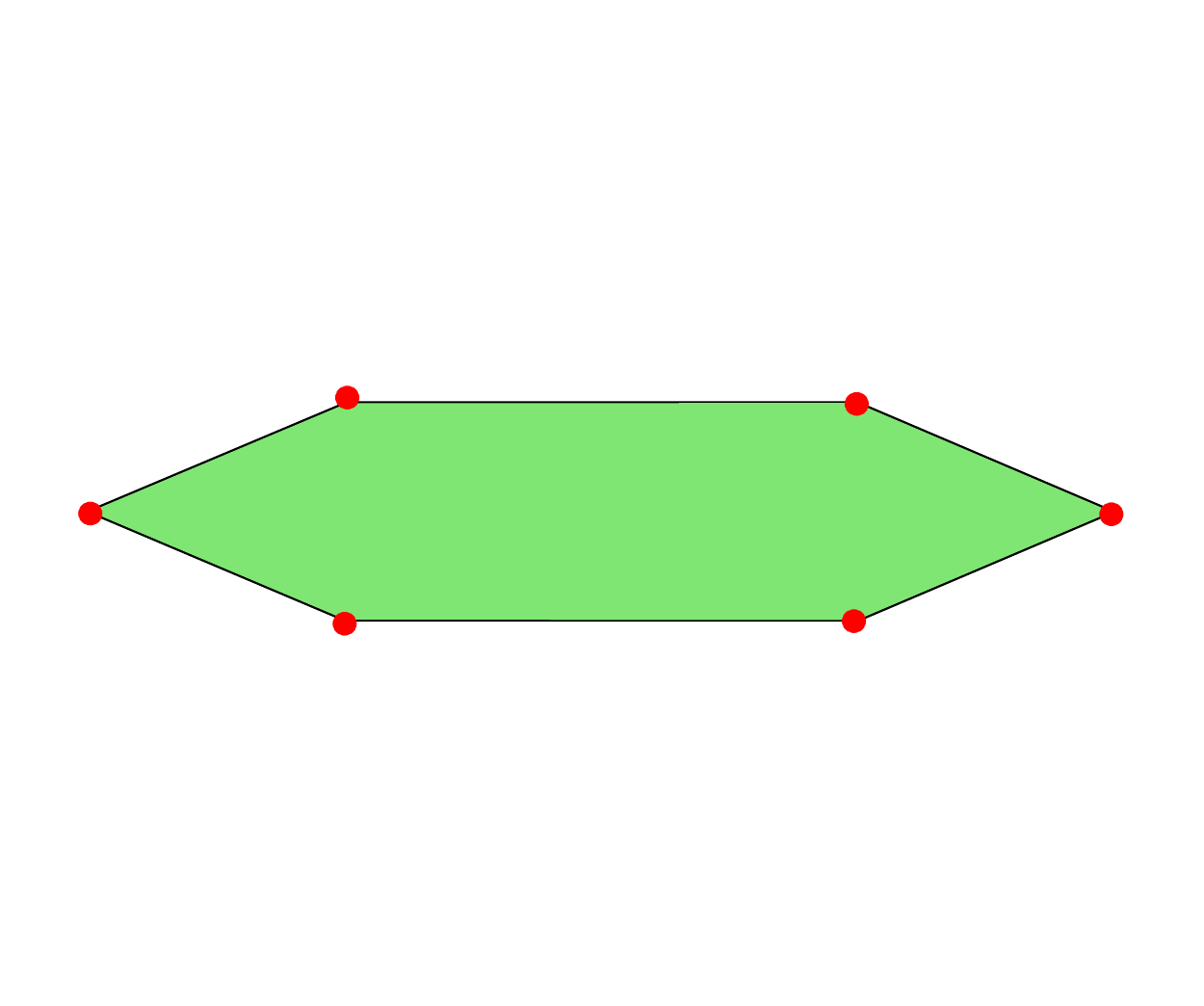}}
\caption{The hexagons $H_0, H_1, H_2$.}
  \label{fig:collapsehexagon} 
\end{figure} 
We design two experiments to check the condition number of the stiffness
matrices of the four VEMs. 

Firstly, we refer to the ``collapsing polygons'' experiment in
\cite{Mascotto2018} and consider a sequence of hexagons
$\{H_i\}_{i=0}^{\infty}$, where the vertices of $H_i$ are given by
$A_i = (1, 0)$, $B_i = (0.5, a_i)$, $C_i = (-0.5, a)$, $D_i=(-1, 0)$, 
$E_i = (-0.5, -a_i)$, and $F_i = (0.5, -a_i)$,
where $a_i = \frac{\sqrt{3}}{2^{i+1}}$. The hexagons $H_0$, $H_1$ and $H_2$ are drawn
in Fig.~\ref{fig:collapsehexagon}. 
As shown in Fig.~\ref{fig:collapsehexagon_conditionnumber} for $k=8$ and $k=10$, the condition numbers of stiffness
matrices of the VEMs without extrinsic stabilization are smaller than those of the standard methods when $i$ is large.
\begin{figure}[htbp]
\subfigure[$k = 8$]{
\begin{minipage}[t]{0.45\linewidth}
\centering
\includegraphics*[width=2in]{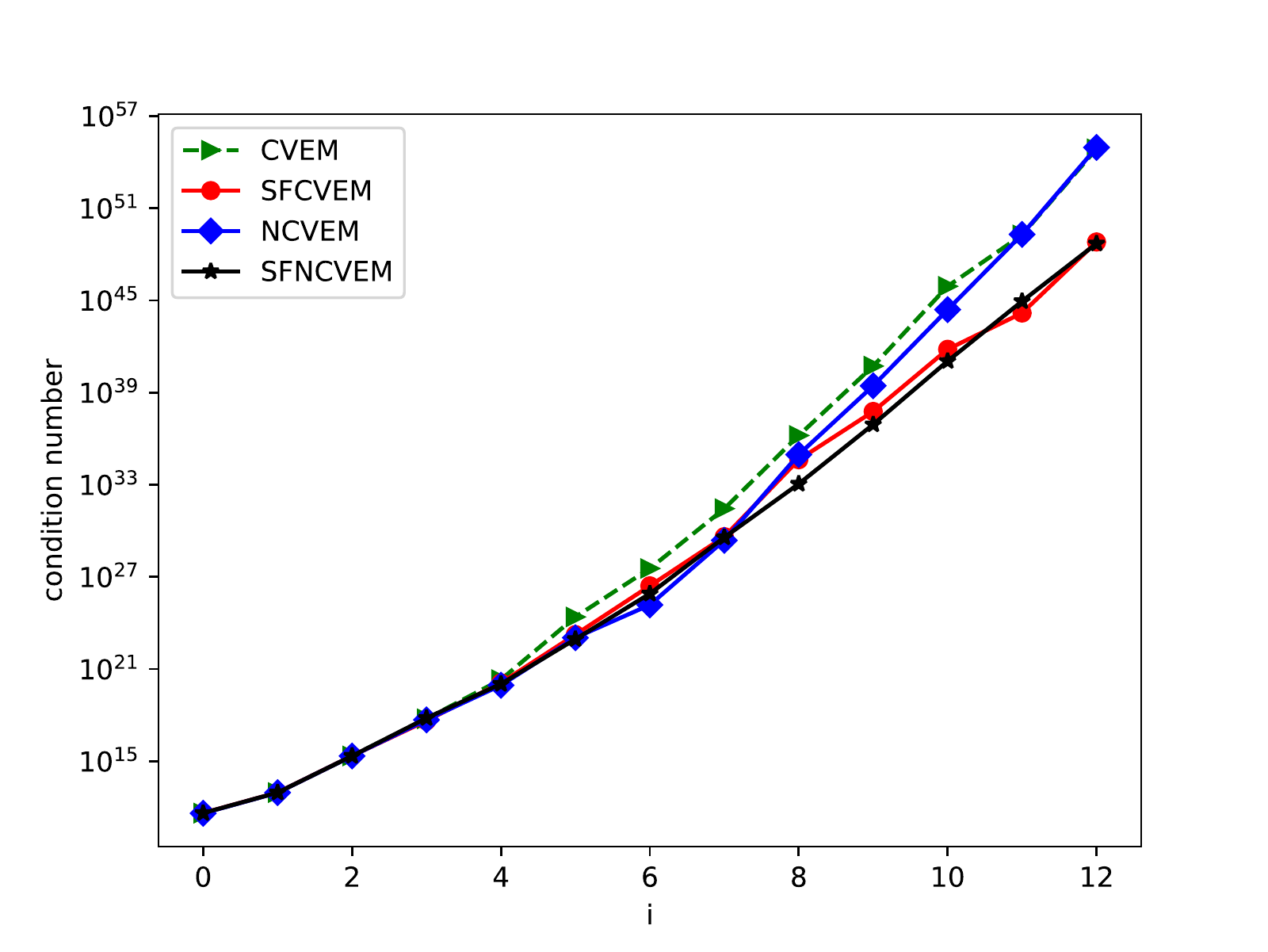}
\end{minipage}}
\quad \quad
\subfigure[$k = 10$]
{\begin{minipage}[t]{0.45\linewidth}
\centering
\includegraphics*[width=2in]{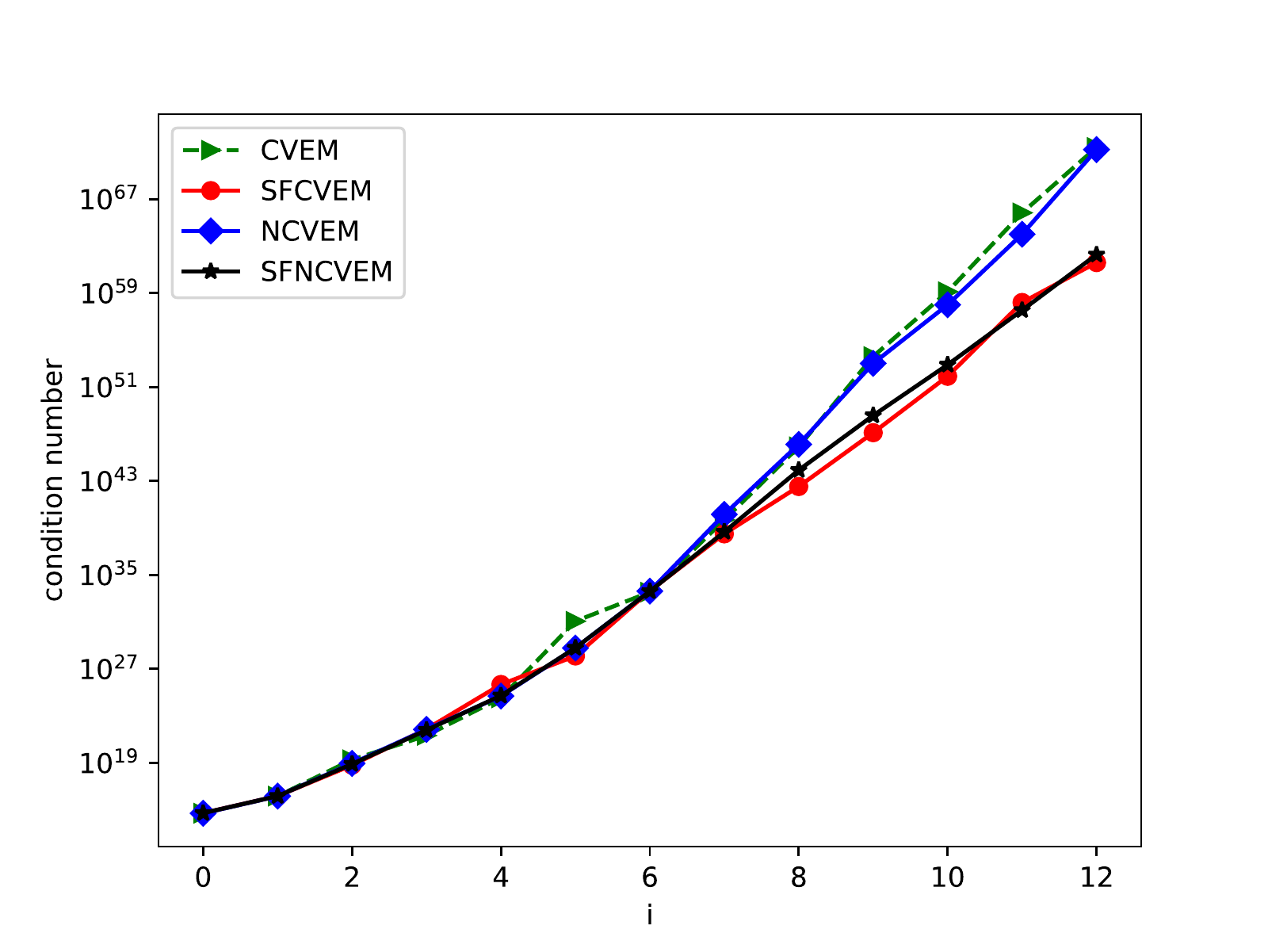}
\end{minipage}}
\caption{The condition number of stiffness matrix of four VEMs on
$\{H_i\}_{i=0}^{12}$.}
\label{fig:collapsehexagon_conditionnumber}
\end{figure}


Secondly, we do the patch test for the Laplace equation, i.e. problem \eqref{eq:ellipitc2ndproblem} with $\alpha = 0$ and $f=0$, but the Dirichlet boundary condition is nonhomogeneous. Take the exact solution $u=1+x+y$. 
Let $h_x$ and $h_y$ be mesh size in the $x$-direction and
$y$-direction respectively.
We examine the behavior of error $\|u - u_h\|_0$ of the four VEMs in the
following three cases:
\begin{enumerate}[(1)]
\item Mesh in Fig.~\ref{fig:mesh}(a): Fix $h_x=h_y=0.2$ but vary $k= 1, 2, \ldots, 10$;
\item Mesh in Fig.~\ref{fig:mesh}(a): Fix $k=3$ but vary $h_x=h_y=2^{-i}$ for $i=1,\ldots, 5$;
\item Mesh in Fig.~\ref{fig:polymesh}: Fix $k=3$ and $h_x=0.2$, but vary $h_y=2^{-i}$ for $i=1,\ldots, 8$.
\end{enumerate}
\begin{figure}[htbp]
\centering
\subfigure{\includegraphics[width=2in]{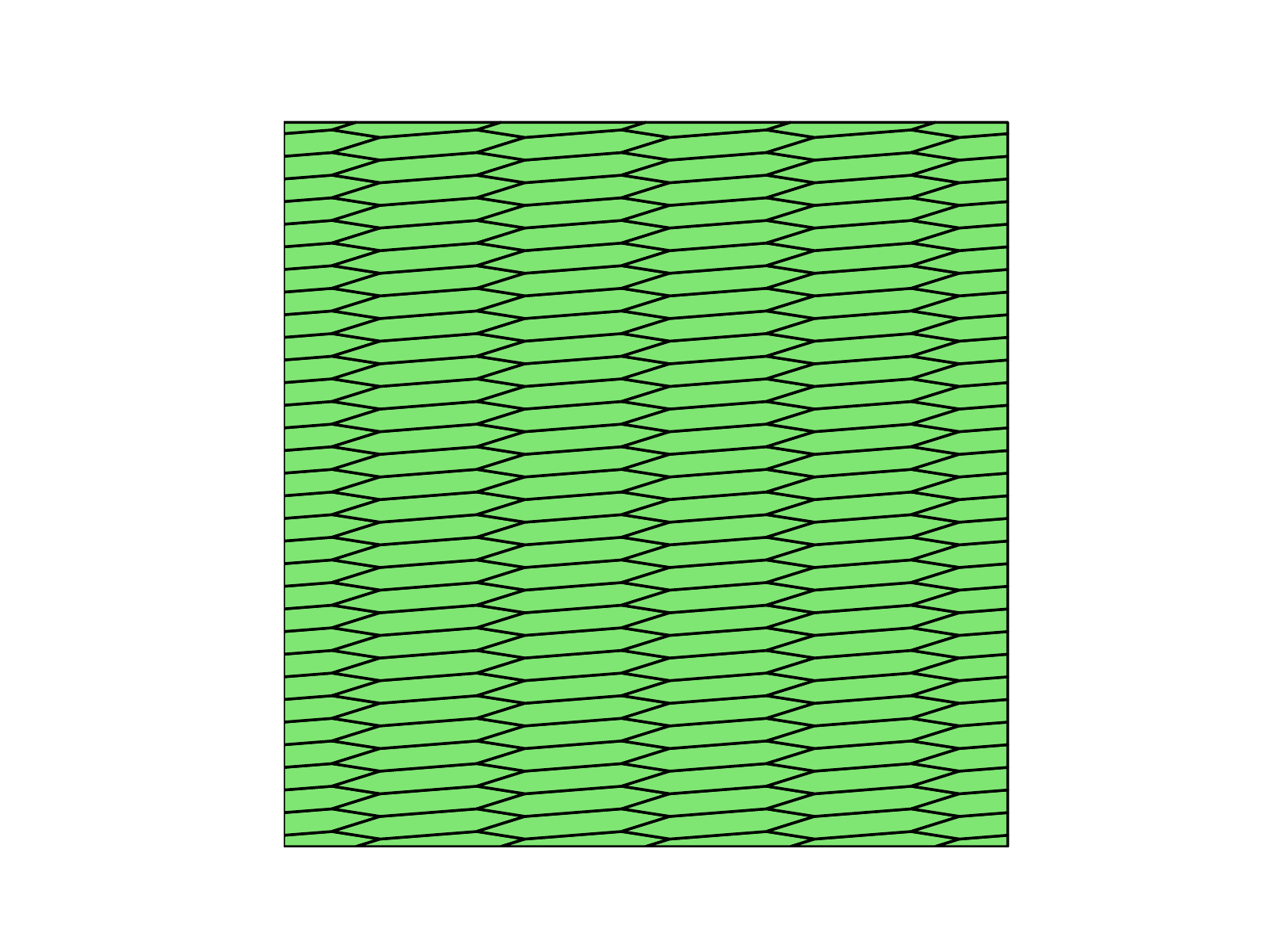}}
\caption{The mesh of domain $(0, 1)\times(0, 1)$ with
$h_x=0.2, h_y = 0.03125$.}
  \label{fig:polymesh} 
\end{figure}
The errors of the four methods shown in Fig.~\ref{fig:patchtest} are similar. 
Since the error in the patch test
grows as the condition number of the stiffness matrix grows, 
the condition numbers of the stiffness matrix obtained by
four methods are comparable.
\begin{figure}[htbp]
\centering
\subfigure{\includegraphics[width=1.8in]{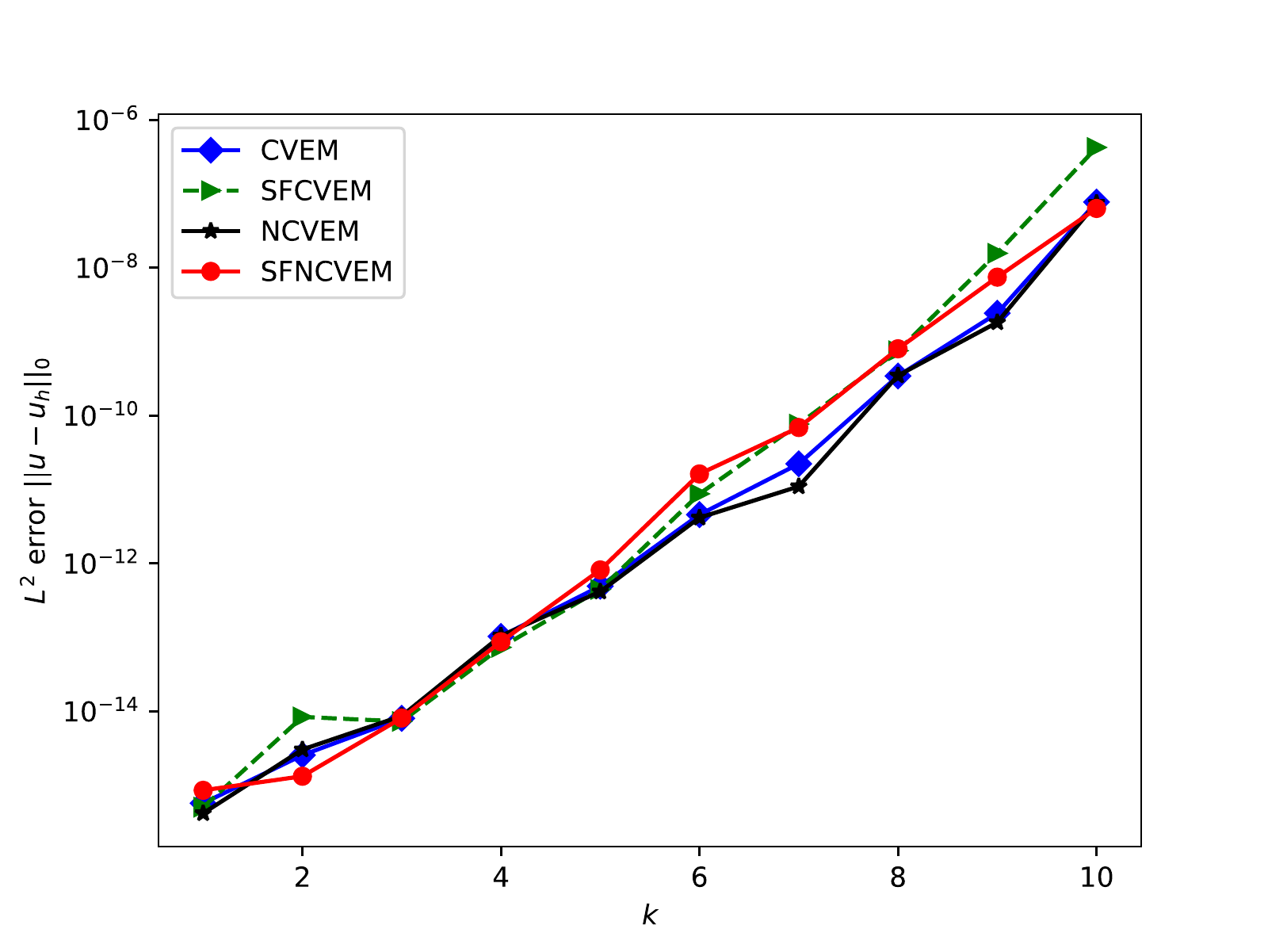}}
\subfigure{\includegraphics[width=1.8in]{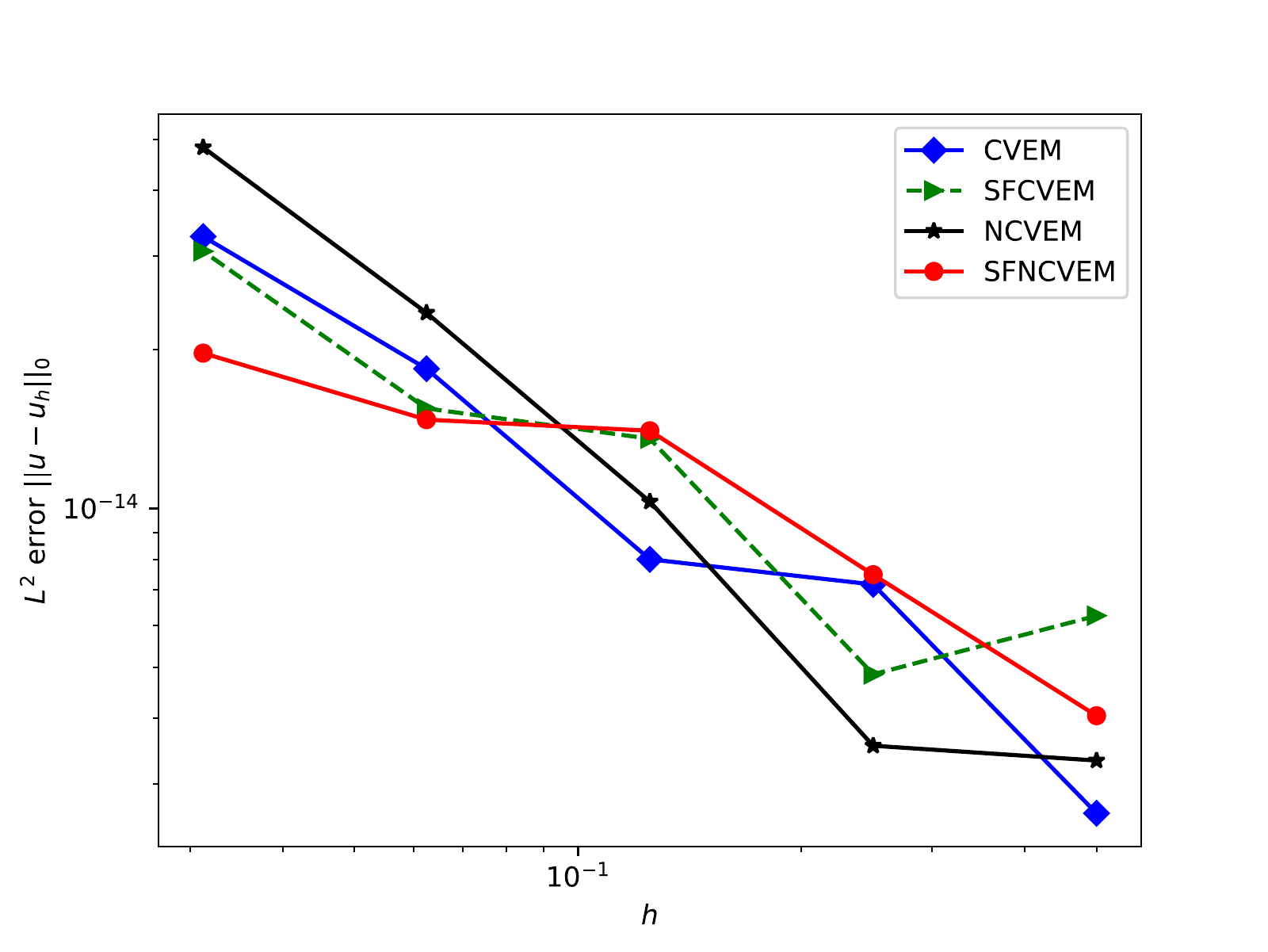}}
\subfigure{\includegraphics[width=1.8in]{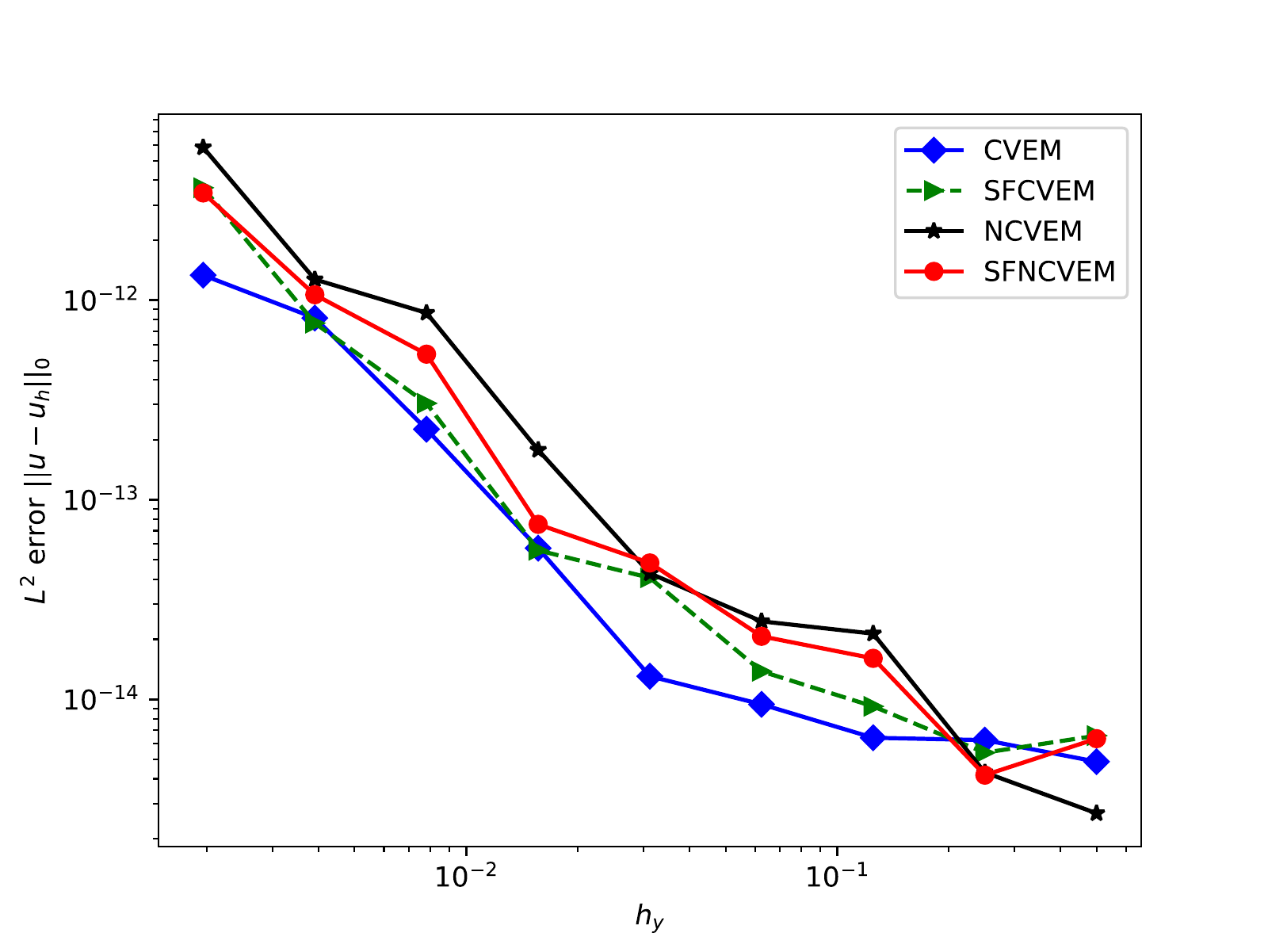}}
\caption{The $L^2$ error of patch test.}
\label{fig:patchtest} 
\end{figure}

\bibliographystyle{abbrv}
\bibliography{./paper}
\end{document}